\documentclass[final,leqno]{siamltex704}
\usepackage{amsmath}
\usepackage{graphicx}
 \usepackage[notcite,notref]{showkeys}
\usepackage{mathrsfs}
\usepackage{color}
\usepackage{float}
\usepackage{amsfonts,amssymb}
\usepackage{dsfont}
\usepackage{pifont}
\usepackage{wrapfig} 
\usepackage{hyperref}
\usepackage{multirow}
\numberwithin{equation}{section}
\def\3bar{{|\hspace{-.02in}|\hspace{-.02in}|}}
\def\E{{\mathcal{E}}}
\def\T{{\mathcal{T}}}

\def\pT{{\partial T}}

\def\W{{\mathcal{W}}}

\def\bn{{\mathbf{n}}}

\def\bx{{\mathbf{x}}}

\def\bF{{\textbf{F}}}
\def\ljump{{[\![}}
\def\rjump{{]\!]}}

\def\bbeta{{\boldsymbol{\beta}}}

\newtheorem{algorithm}{Primal-Dual Weak Galerkin Algorithm}[section]

\title{A primal-dual finite element method for first-order transport problems}

\author{Chunmei Wang \thanks{Department of Mathematics \& Statistics, Texas Tech University, Lubbock, TX 79409, USA (chunmei.wang@ttu.edu). The research of Chunmei Wang was partially supported by National Science Foundation Award DMS-1849483.} \and Junping Wang\thanks{Division of Mathematical
Sciences, National Science Foundation, Alexandria, VA 22314
(jwang@nsf.gov). The research of Junping Wang was supported by the NSF IR/D program, while working at National Science Foundation. However, any opinion, finding, and conclusions or recommendations expressed in this material are those of the author and do not necessarily reflect the views of the National Science Foundation.}
}

\begin{document}

\maketitle
\begin{abstract}
This article devises a new numerical method for first-order transport problems by using the primal-dual weak Galerkin (PD-WG) finite element method recently developed in scientific computing. The PD-WG method is based on a variational formulation of the modeling equation for which the differential operator is applied to the test function so that low regularity for the exact solution of the original equation is sufficient for computation. The PD-WG finite element method indeed yields a symmetric system involving both the original equation for the primal variable and its dual for the dual variable (also known as Lagrangian multiplier). For the linear transport problem, it is shown that the PD-WG method offers numerical solutions that conserve mass locally on each element. Optimal order error estimates in various norms are derived for the numerical solutions arising from the PD-WG method with weak regularity assumptions on the modelling equations.  A variety of numerical results are presented to demonstrate the accuracy and stability of the new method.
\end{abstract}

\begin{keywords}
primal-dual finite element method, weak Galerkin, transport equation, discrete weak gradient, polytopal partitions, weak regularity,  conservative methods.
\end{keywords}

\begin{AMS}
Primary, 65N30, 65N15, 65N12; Secondary, 35L02, 35F15, 35B45
\end{AMS}

\pagestyle{myheadings}

\section{Introduction}
In this paper we are concerned with the development of new numerical methods for first-order linear convection equations in divergence form.
For simplicity, consider the model problem that seeks an unknown function $u$ satisfying
\begin{equation}\label{model}
\begin{split}
\nabla\cdot (\bbeta u)+c u=&f, \qquad \text{in }\Omega,\\
u=&g, \qquad  \text{on }\Gamma_-,
\end{split}
\end{equation}
where $\Omega$ is an open bounded and connected domain in $\mathbb R^d \ (d=2, 3)$ with Lipschitz continuous boundary $\partial \Omega$,   $\Gamma_-$ is the inflow portion of the boundary defined by
$$
\Gamma_- = \{\bx\in \partial\Omega: \ \bbeta(\bx) \cdot \bn < 0\},
$$
where $\bn$ is the unit outward normal vector to the boundary $\partial \Omega$ at a given point. Assume that  the convection vector satisfies $\bbeta=(\beta_1, \cdots, \beta_d)\in [L^{\infty}(\Omega)]^d$ and is locally as smooth as $W^{1,\infty}$, the reaction coefficient $c\in L^{\infty} (\Omega)$, the load function $f\in L^2(\Omega)$, and the inflow boundary data $g\in L^2(\Gamma_-)$.

First-order linear partial differential equations (PDEs) of hyperbolic-type are called transport equations or linear convection equations. Problems of hyperbolic-type arise in many areas of science and engineering, such as fluid dynamics and neutron transport. In the past several decades, there have been  increasing research activities devoted to the development of accurate and efficient numerical methods for hyperbolic problems. Due to the localized phenomena, such as propagating discontinuities and sharp transition layers, it has been a challenging task to develop effective numerical methods for hyperbolic problems. Due to largely the fact that linear hyperbolic PDEs admit discontinuous solutions for non-smooth boundary data, it is difficult to develop numerical methods that provide high-order accurate approximations in regions of smooth solution as well as sharp resolution of discontinuity while avoiding spurious oscillations at discontinuities \cite{bd1999}. Linear hyperbolic equations also form prototype equations for general hyperbolic equations, such as systems of nonlinear conservation laws \cite{l1992} or transport equations in phase space \cite{lm1993}.  It has been shown that successful numerical methods for linear convection equations can be used as  building blocks for the numerical solution of complex hyperbolic PDEs \cite{l1992}. In literature, a series of numerical methods have been developed for linear transport equations, including the streamline-upwind Petrov-Galerkin method \cite{eehj1996}, the residual distribution framework \cite{jnp1984, eehj1996, a2001}, the least-squares finite element methods \cite{cj1988, bc2001, hjs2002, smmo2004, hjs2002, b1999, smmo2004}, the stabilized finite element methods \cite{burman2014}, and various discontinuous Galerkin finite element methods \cite{rh, lr, c1999, cs, hss, j2004, eg, bh, b, be, bs, bs2007}.

Most of the existing study for the linear transport equation \eqref{model} typically assumes certain coercivity condition on the convection vector $\bbeta$ and the reaction coefficient $c$ in the form of $c+\frac12 \nabla\cdot\bbeta \ge \alpha_0$ or alike for some fixed positive number $\alpha_0$. In practical applications, these conditions are very restrictive and often rule out many important physics such as those with compressible flow fields and exothermic reactions \cite{burman2014}. One such exception is the stabilized finite element method developed by Burman in  \cite{burman2014}, in which the convection vector is assumed to satisfy $\bbeta\in [W^{2,\infty}(\Omega)]^d$ and $c\in W^{1,\infty}(\Omega)$.

The objective of this paper is to develop a new numerical method for the linear convection problem \eqref{model} for piecewise smooth convection vector $\bbeta$ and reaction coefficient $c$ without assuming any coercivity conditions on the equation.
Our new numerical schemes will be devised by following the primal-dual weak Galerkin (PD-WG) framework introduced and studied in \cite{ww2016, ww2017, ww2018, wz2019, w2018}. The PD-WG finite element method was originally developed for the second order elliptic PDEs in non-divergence form through
a constraint optimization approach of the problem in which the constraint was given by a straightforward discretization of the PDE and the optimization was imposed to minimize the ``discontinuity" of the approximating functions. The resulting Euler-Lagrange formulation reveals a symmetric problem involving both the primal (original) equation and the dual (adjoint) equation integrated through various stabilizers designed to provide certain ``weak continuity or smoothness". The framework of the primal-dual methods in numerical PDEs was also developed by Burman \cite{Burman2013, burman2014} in other finite element contexts and was named {\em stabilized finite element methods}. The PD-WG finite element methods have shown great promises for PDE problems where no traditional variational formulations are readily available for analysis and discretization.

Let us briefly discuss the philosophy and the key ingredient in the PD-WG finite element method for the first-order hyperbolic problem   \eqref{model}. First of all, the solution of the model problem (\ref{model}) can be characterized by seeking $u \in L^2(\Omega)$ such that \begin{equation}\label{weakform}
(u, \bbeta \cdot \nabla\sigma- c \sigma)=\langle g, \bbeta \cdot \bn \sigma\rangle_{\Gamma_-}-(f,\sigma), \qquad \forall \sigma\in H_{0,\Gamma_+}^1(\Omega),
\end{equation}
where  $\Gamma_+=\partial \Omega \setminus \Gamma_-$ is the outflow boundary satisfying $\bbeta \cdot \bn \geq 0$, and $H_{0,\Gamma_+}^1(\Omega)$ is the subspace of $H^1(\Omega)$ with vanishing boundary value on $\Gamma_+$; i.e.,
$$
H_{0,\Gamma_+}^1(\Omega)=\{v \in H^1(\Omega): v=0 \ \text{on}\ \Gamma_+\}.
$$
Secondly, using the weak gradient operator $\nabla_w$ \cite{wy3655},  we may reformulate \eqref{weakform} as follows:
\begin{equation}\label{weakform-02}
(u, \bbeta \cdot \nabla_w\{\sigma\}- c \sigma)=\langle g, \bbeta \cdot \bn \sigma\rangle_{\Gamma_-}-(f,\sigma), \qquad \forall \sigma\in H_{0,\Gamma_+}^1(\Omega),
\end{equation}
where $\{\sigma\} = \{\sigma|_T, \sigma|_\pT\}$ is understood as a weak function in the WG context. The weak function is discretized by piecewise polynomials in each element $T$ as well as on its boundary $\pT$. The weak gradient operator $\nabla_w$ is then approximated by vector-valued polynomials, denoted as $\nabla_{w, h}$. The weak formulation (\ref{weakform-02}) can then be approximated by seeking $u_h\in M_h$ (i.e., trial space) such that
\begin{equation}\label{EQ:10-12-2015:01}
(u_h, \bbeta \cdot \nabla_{w, h}{\sigma}- c \sigma)=\langle g, \bbeta \cdot \bn \sigma\rangle_{\Gamma_-}-(f,\sigma_0), \qquad \forall \sigma\in W_h^{0,\Gamma_+},
\end{equation}
where $W_h^{0,\Gamma_+}$ is a test space for the weak functions with vanishing boundary value on $\Gamma_+$. However, the problem (\ref{EQ:10-12-2015:01}) is not well-posed unless the {\em inf-sup} condition of Babu\u{s}ka \cite{babuska} is satisfied. The primal-dual scheme overcomes this difficulty by coupling \eqref{EQ:10-12-2015:01} with its dual equation which seeks $\lambda_h\in W_h^{0,\Gamma_+}$ satisfying
\begin{equation}\label{EQ:09-12-2018:01}
(v, \bbeta\cdot \nabla_{w, h} \lambda_h-c\lambda_h) =0, \qquad
\forall v \in M_h.
\end{equation}
A formal coupling between \eqref{EQ:10-12-2015:01} and \eqref{EQ:09-12-2018:01} makes an effective numerical scheme through the use of a stabilizer, denoted by $s(\lambda, \sigma)$. This stabilizer measures the level of ``continuity" of $\sigma \in W_{h}$ in the sense that $\sigma \in W_{h}$ is of classical $C^0$-conforming if and only if $s(\sigma, \sigma)=0$. The resulting primal-dual weak Galerkin method for solving the hyperbolic model problem (\ref{model}) seeks $u_h\in M_h$ and $\lambda_h\in W_h^{0,\Gamma_+}$, such that
\begin{equation}\label{primal-dual-wg}
 \left\{\begin{split}
s(\lambda_h, \sigma)+(u_h, \bbeta \cdot \nabla_{w, h}{\sigma}- c \sigma)=&\langle g, \bbeta \cdot \bn \sigma\rangle_{\Gamma_-}-(f,\sigma_0), \  \forall \sigma\in W_h^{0,\Gamma_+},\\
(v, \bbeta\cdot \nabla_{w, h} \lambda_h-c\lambda_h) =&0, \qquad\qquad\qquad\qquad\qquad
\forall v \in M_h,
 \end{split} \right.
\end{equation}
where $s(\cdot, \cdot)$ is known as the stabilizer or smoother which enforces a certain weak continuity for the numerical Lagrangian multiplier $\lambda_h$ in the weak finite element space $W_h^{0,\Gamma_+}$.

In this paper, we show that the PD-WG finite element method \eqref{primal-dual-wg} has one and only one solution if the linear convection problem \eqref{model} admits at most one solution. The numerical scheme  \eqref{primal-dual-wg} will be shown to be conservative locally on each element in the sense that there exists a numerical solution $\tilde u_h$ and a numerical flux $\textbf{F}_h$ satisfying (see Theorem \ref{THM:conservation})
\begin{equation*}
\int_{\partial T}\textbf{F}_h \cdot \bn ds+\int_T c\tilde u_h dT=\int_T fdT,\qquad T\in \T_h.
\end{equation*}
Some optimal order error estimates will be derived for the numerical solution arising from the PD-WG scheme \eqref{primal-dual-wg} under ultra-weak assumptions on the convection vector $\bbeta$ and the reaction coefficient $c$.

The paper is organized as follows. In Section \ref{Section:DWG} we shall briefly review the weak gradient operator as well as its discretization.
In Section \ref{Section:WGFEM}, we give a detailed presentation on the primal-dual weak Galerkin algorithm for the linear hyperbolic problem (\ref{model}). A discussion on the solvability (i.e., the solution existence and uniqueness) of the PD-WG scheme is given in Section \ref{Section:EU}. In Section \ref{Section:MC} we show that the PD-WG method offers numerical solutions that conserve mass locally on each element. Next in Section \ref{Section:EE}, we shall derive an error equation for the PD-WG approximations. In Section \ref{Section:TechEst}, we establish some technical estimates/results useful in convergence analysis. Section \ref{Section:ErrorEstimate} is devoted to the establishment of an optimal order error estimate for the PD-WG approximations in some discrete Sobolev norms. In Section \ref{Section:L2Error}, an error estimate in a weak $L^2$ topology is derived based on a local $H^{1}$-regularity assumption for the dual problem. In Section \ref{Section:Numerics}, a series of numerical results are reported to demonstrate the effectiveness and accuracy of the PD-WG method developed in the previous sections.

Throughout the paper, we follow the usual notations for Sobolev spaces and norms. For any open bounded domain $D\subset \mathbb{R}^d$ with Lipschitz continuous boundary, denote by $\|\cdot\|_{s, D}$ and $|\cdot|_{s, D}$ the norm and semi-norm in the Sobolev space $H^s(D)$ for any $s\ge 0$, respectively. The inner product in $H^s(D)$ is denoted by $(\cdot, \cdot)_{s,D}$. The space $H^0(D)$ coincides with $L^2(D)$, for which the norm and the inner product are denoted by $\|\cdot\|_{D}$ and $(\cdot, \cdot)_{D}$, respectively. When $D=\Omega$, or when the domain of integration is clear from the context,  the subscript $D$ is dropped in the norm and the inner product notations.


\section{Discrete Weak Gradient}\label{Section:DWG}
The principle differential operator in the weak formulation (\ref{weakform}) for the linear convection equation in divergence form (\ref{model}) is given by the gradient operator. In this section we shall review the weak gradient operator as well as its discrete version introduced originally in \cite{wy3655}.

Let $T$ be a polygonal or polyhedral domain with boundary $\partial T$. By a weak function on $T$ we mean a pair $v=\{v_0,v_b\}$ such that $v_0\in L^2(T)$ and $v_b\in L^{2}(\partial T)$. The components $v_0$ and $v_b$ can be understood as the value of $v$ in the interior and on the boundary of $T$, respectively. Note that $v_b$ is not necessarily the trace of $v_0$ on $\partial T$, though taking $v_b$ as the trace of $v_0$ on $\partial T$ is a feasible option. Denote by $\W(T)$ the space of weak functions on $T$; i.e.,
\begin{equation*}\label{2.1}
\W(T)=\{v=\{v_0, v_b\}: v_0 \in L^2(T), v_b \in L^{2}(\partial T)\}.
\end{equation*}

The weak gradient of $v\in \W(T)$, denoted by $\nabla_w v$, is defined as a continuous linear functional in the Sobolev space $[H^1(T)]^d$ satisfying
\begin{equation*}
(\nabla_w  v,\boldsymbol{\psi})_T=-(v_0,\nabla \cdot \boldsymbol{\psi})_T+\langle v_b,\boldsymbol{\psi}\cdot \textbf{n}\rangle_{\partial T},  \qquad \forall \boldsymbol{\psi}\in [H^1(T)]^d.
\end{equation*}

Denote by $P_r(T)$ the space of polynomials on $T$ with degree $r$ and less. A discrete version of $\nabla_{w} v$  for $v\in \W(T)$, denoted by $\nabla_{w, r, T}v$, is defined as the unique polynomial-valued vector in $[P_r(T) ]^d$ satisfying
\begin{equation}\label{disgradient}
(\nabla_{w, r, T} v, \boldsymbol{\psi})_T=-(v_0, \nabla \cdot \boldsymbol{\psi})_T+\langle v_b, \boldsymbol{\psi} \cdot \textbf{n}\rangle_{\partial T}, \quad\forall\boldsymbol{\psi}\in [P_r(T)]^d,
\end{equation}
which, from the usual integration by parts, gives
\begin{equation}\label{disgradient*}
(\nabla_{w, r, T} v, \boldsymbol{\psi})_T= (\nabla v_0, \boldsymbol{\psi})_T-\langle v_0- v_b, \boldsymbol{\psi} \cdot \textbf{n}\rangle_{\partial T}, \quad\forall\boldsymbol{\psi}\in [P_r(T)]^d,
\end{equation}
provided that $v_0\in H^1(T)$.

\section{Primal-Dual Weak Galerkin Algorithm}\label{Section:WGFEM}
Let ${\cal T}_h$ be a partition of the domain $\Omega \subset {\mathbb R}^d (d=2, 3)$ into polygons in 2D or polyhedra in 3D which is shape regular described as in \cite{wy3655}. Denote by ${\mathcal E}_h$ the set of all edges or flat faces in ${\cal T}_h$ and ${\mathcal E}_h^0={\mathcal E}_h \setminus \partial\Omega$ the set of all interior edges or flat faces. Denote by $h_T$ the meshsize of $T\in {\cal T}_h$ and
$h=\max_{T\in {\cal T}_h}h_T$ the meshsize of the partition ${\cal T}_h$.

For any integer $j\geq 0$, denote by $W_j(T)$ the local space of discrete weak functions; i.e.,
$$
W_j(T)=\{\{\sigma_0,\sigma_b\}:\sigma_0\in P_j(T),\sigma_b\in P_j(e), e\subset \partial T\}.
$$
Patching $W_j(T)$ over all the elements $T\in {\cal T}_h$ through a common value $v_b$ on the interior interface $\E_h^0$ yields a global weak finite element space $W_{j,h}$. Let $W_{j,h}^{0, \Gamma_+}$ be the subspace of $W_{j,h}$ with vanishing boundary values on $\Gamma_+$; i.e.,
$$
W_{j,h}^{0, \Gamma_+}=\{\{\sigma_0, \sigma_b\}\in W_{j,h}: \sigma_b|_{e}=0, e\subset \Gamma_+\}.
$$
For any integer $k\ge 1$, let $M_{k-1,h}$ be the space of piecewise polynomials of degree $k-1$; i.e.,
$$
M_{k-1,h}=\{w: w|_T\in P_{k-1}(T), \forall T\in {\cal T}_h\}.
$$

The discrete weak gradient $\nabla_{w,r,T}$ shall be taken in the polynomial subspace $P_r(T)$ with $r=k-1$. For simplicity of notation and without confusion, denote by $\nabla_{w}\sigma$ the discrete weak gradient
$\nabla_{w, k-1, T}\sigma$ for any $\sigma\in W_{j,h}$ computed by
(\ref{disgradient}) on each element $T$; i.e.,
$$
(\nabla_{w}\sigma)|_T= \nabla_{w, k-1, T}(\sigma|_T), \qquad \sigma\in W_{j, h}.
$$

For any $\rho, \sigma\in W_{j,h}$ and $v\in M_{k-1,h}$, we introduce the
following bilinear forms
\begin{eqnarray}
s(\rho, \sigma)&=&\sum_{T\in {\cal T}_h}s_T(\rho, \sigma),\label{stabilizer}
\\
b(v, \sigma)&=&\sum_{T\in {\cal T}_h}b_T(v, \sigma),\label{b-form}
\end{eqnarray}
where
\begin{eqnarray}\label{stabilizer-local}
s_T( \rho, \sigma)&=&h_T^{-1}\int_{\partial T} (\rho_0-\rho_b)(\sigma_0-\sigma_b)ds\\
& & +\tau \int_{ T} (\bbeta\cdot\nabla\rho_0-c\rho_0)(\bbeta\cdot\nabla\sigma_0-c\sigma_0)dT,\nonumber\\
b_T(v, \sigma)&=&(v, \bbeta \cdot \nabla_w \sigma -c \sigma_0)_T,\label{b-form-local}
\end{eqnarray}
with $\tau\geq 0$ being a parameter.

The numerical scheme for the linear convection equation (\ref{model}) in divergence form based on the variational formulation (\ref{weakform}) in the framework of primal-dual approach is given as follows:
\begin{algorithm}
Find $(u_h;\lambda_h)\in M_{k-1,h} \times W_{j,h}^{0, \Gamma_+}$ such that
\begin{eqnarray}\label{32}
s(\lambda_h, \sigma)+b(u_h,\sigma)&=& \sum_{e\subset \Gamma_-}\langle \sigma_b, \bbeta \cdot \bn g  \rangle_{e}-(f,\sigma_0), \qquad \forall\sigma\in W_{j,h}^{0, \Gamma_+},\\
b(v,\lambda_h)&=&0,\qquad \qquad\qquad\qquad \qquad \qquad \forall v\in M_{k-1,h}.\label{2}
\end{eqnarray}
\end{algorithm}

\section{Solution Existence and Uniqueness}\label{Section:EU}

 The following is the adjoint problem for the linear transport equation: For a given $\theta \in L^2(\Omega)$, find $\Psi$ such that
\begin{eqnarray}\label{mo}
\bbeta \cdot \nabla \Psi-c\Psi&=&\theta \qquad \mbox{ in } \Omega,\\
\Psi & = & 0 \qquad\mbox { on } \Gamma_+. \label{mo-2}
\end{eqnarray}

The adjoint problem (\ref{mo})-(\ref{mo-2}) is said to have the $H^{\gamma}$-regularity with some parameter  $\gamma \in (0,1]$
if it has a solution $\Psi\in H^\gamma(\Omega)$ satisfying
\begin{equation}\label{reg}
\|\Psi \|_{\gamma} \leq C\|\theta\|,
\end{equation}
with a generic constant $C$.

For simplicity, we now introduce several $L^2$ projection operators. On each element $T$, denote by $Q_0$ the $L^2$ projection operator onto $P_j(T)$. For each edge or face $e\subset\partial T$, denote by $Q_b$ the $L^2$ projection operator onto $P_{j}(e)$. For any $w\in H^1(\Omega)$,  we use $Q_h w$ to denote the $L^2$ projection of $w$ in the finite element space $W_{j,h}$ given by
$$
Q_hw=\{Q_0w,Q_bw\}
$$
on each element $T$. ${\cal Q}_h$ is used to denote the $L^2$ projection operator onto the space $M_{k-1,h}$.

\begin{lemma}\label{Lemma5.1} \cite{wy3655} For $j\ge k-1$, the $L^2$ projection operators $Q_h$ and ${\cal Q}_h$ satisfy the following commutative property:
\begin{equation}\label{l}
\nabla_{w}(Q_h w) = {\cal Q}_h(\nabla w), \qquad \forall w\in H^1(T).
\end{equation}
\end{lemma}

For the convenience of analysis, in what follows of this paper, we assume that the convection vector $\bbeta$ and the reaction coefficient $c$ are both piecewise constants with respect to the partition $\T_h$. However, all the analysis and results can be generalized and extended to piecewise smooth cases for the convection vector $\bbeta$ and the reaction coefficient $c$.

\begin{theorem}\label{thmunique1} Assume that the linear transport problem \eqref{model} and the dual problem \eqref{mo}-\eqref{mo-2} have unique solutions. If $\tau>0$, then the PD-WG algorithm (\ref{32})-(\ref{2}) has one and only one solution for $j=k-1, k$.
\end{theorem}

\begin{proof} It suffices to show that the homogeneous problem of (\ref{32})-(\ref{2}) has only the trivial solution. To this end, we assume $f=0$ and $g=0$. By letting $v=u_h$ and $\sigma=\lambda_h$ in (\ref{32}) and (\ref{2}), we obtain $s(\lambda_h, \lambda_h)=0$, which implies $\lambda_0=\lambda_b$ and $\bbeta\cdot\nabla\lambda_0-c\lambda_0=0$ on each $\partial T$. It follows that
\begin{equation}\label{aa-001}
\bbeta\cdot\nabla \lambda_0-c\lambda_0=0\quad  \mbox{strongly in $\Omega$}.
\end{equation}
Since, on the portion $\Gamma_+$ of the domain boundary, $\lambda_0=\lambda_b\equiv 0$ holds true, then from the solution uniqueness for the adjoint problem \eqref{mo}-\eqref{mo-2} we have $\lambda_0 \equiv 0$ in $\Omega$.
It follows from $\lambda_0=\lambda_b$ on each  $\partial T$ that $\lambda_b\equiv0$ and thus $\lambda_h\equiv 0$.

Next, substituting $\lambda_h\equiv 0$ into (\ref{32}) gives
$$
b(u_h,\sigma)=0,  \qquad \forall\sigma\in W_{j,h}^{0, \Gamma_+}.
$$
From the usual integration by parts and (\ref{disgradient}) we obtain
\begin{equation}\label{eee}
\begin{split}
0=& b(u_h,\sigma)\\
 =& \sum_{T\in {\cal T}_h} (u_h, \bbeta \cdot \nabla_w\sigma-c\sigma_0)_T\\
= & \sum_{T\in {\cal T}_h} -(\sigma_0,  \nabla\cdot(\bbeta u_h))_T+\langle \sigma_b, \bbeta u_h\cdot \bn\rangle_{\partial T}-(u_h, c\sigma_0)_T\\
= & -\sum_{T\in {\cal T}_h}(\sigma_0, \nabla\cdot(\bbeta u_h)+cu_h)_T+\sum_{e\subset {\cal E}_h \setminus \Gamma_+}\langle \sigma_b, \ljump\bbeta u_h\cdot \bn\rjump\rangle_{e},\\
\end{split}
\end{equation}
where we have used $\sigma_b=0$ on $\Gamma_+$ on the last line, and $\ljump\bbeta u_h\cdot \bn\rjump$ is the jump of $\bbeta u_h\cdot \bn$ on $e\subset {\cal E}_h \setminus \Gamma_+$ in the sense that $\ljump\bbeta u_h\cdot \bn\rjump=\bbeta u_h|_{T_1}\cdot \bn_1 + \bbeta u_h|_{T_2}\cdot \bn_2$ for $e=\partial T_1\cap\partial T_2\subset  {\cal E}_h^0$ with $\bn_1$ and $\bn_2$ being the unit outward normal directions to $\partial T_1$ and $\partial T_2$, respectively, and $\ljump\bbeta u_h\cdot \bn\rjump=\bbeta u_h\cdot \bn$ for $e\subset \Gamma_-$. By setting $\sigma_0=-h_T^2 (\nabla \cdot (\bbeta  u_h)+cu_h)$ on each $T\in {\cal T}_h$ and $\sigma_b= h_T \ljump  \bbeta u_h  \cdot \bn\rjump$ on each $e \subset {\cal E}_h \setminus \Gamma_+$, we may rewrite (\ref{eee}) as follows:    \begin{equation*}
\begin{split}
0= \sum_{T\in {\cal T}_h} h_T^2 \|\nabla \cdot (\bbeta  u_h)+cu_h\|_T^2+ \sum_{e \subset {\cal E}_h \setminus \Gamma_+} h_T\|\ljump\bbeta u_h\cdot \bn\rjump\|_e^2,
\end{split}
\end{equation*}
which gives $\nabla \cdot (\bbeta  u_h)+cu_h=0$ on each $T\in {\cal T}_h$,  $\ljump\bbeta u_h\cdot \bn\rjump=0$ on each $e \subset {\cal E}_h^0$, and $\bbeta u_h\cdot \bn=0$ on each $e\subset\Gamma_-$. This implies that  $\nabla \cdot (\bbeta  u_h)+cu_h=0$ in $\Omega$ and $u_h=0$ on $\Gamma_-$. Thus, from the solution uniqueness assumption, we have $u_h\equiv 0$ in $\Omega$. This completes the proof.
\end{proof}

\begin{theorem} Assume that the linear transport problem \eqref{model} and the dual problem \eqref{mo}-\eqref{mo-2} have unique solutions. The following results hold true:
\begin{itemize}
\item if $j=k-1$, then the PD-WG algorithm (\ref{32})-(\ref{2}) has one and only one solution for any non-negative value of the stabilizer parameter $\tau\ge 0$.
\item if additionally the dual problem \eqref{mo}-\eqref{mo-2} has the $H^{\gamma}$- regularity (\ref{reg}) with some $0<\gamma \leq 1$, then the PD-WG algorithm (\ref{32})-(\ref{2}) has a unique solution for any $\tau=0$ and $j=k-1, k$ provided that the meshsize $h<h_0$ holds true for a sufficiently small, but fixed $h_0>0$.
    \end{itemize}
 \end{theorem}

\begin{proof} It suffices to show that the homogeneous problem of (\ref{32})-(\ref{2}) has only the trivial solution. To this end, we assume $f=0$ and $g=0$. By letting $v=u_h$ and $\sigma=\lambda_h$ in (\ref{32}) and (\ref{2}) we arrive at $s(\lambda_h, \lambda_h)=0$,
 which implies $\lambda_0=\lambda_b$ on each $\partial T$ for any $\tau\ge 0$. It follows from (\ref{2}) and (\ref{disgradient*}) that
\begin{equation*}
\begin{split}
0=& b(v,\lambda_h)\\
 =& \sum_{T\in {\cal T}_h} (v, \bbeta \cdot \nabla_w\lambda_h-c\lambda_0)_T\\
= & \sum_{T\in {\cal T}_h}  ( \nabla \lambda_0, \bbeta v)_T-\langle \lambda_0-\lambda_b, \bbeta v\cdot \bn\rangle_{\partial T}-(v, c\lambda_0)_T\\
 = &  \sum_{T\in {\cal T}_h}  (\bbeta \cdot \nabla \lambda_0 -c  {\cal Q}_h  \lambda_0,   v)_T,
\end{split}
\end{equation*}
where we have used $\lambda_0=\lambda_b$ on each  $\partial T$. By taking $v=\bbeta \cdot \nabla \lambda_0 -c {\cal Q}_h  \lambda_0$ we obtain
$$
\bbeta \cdot \nabla \lambda_0 -c {\cal Q}_h \lambda_0=0
$$
on each element $T\in {\cal T}_h$. From $\lambda_0=\lambda_b$ on each  $\partial T$, we have $\lambda_0\in H^1(\Omega)$ so that
\begin{equation}\label{new:001}
\bbeta \cdot \nabla \lambda_0-c\lambda_0=\theta, \quad \text{in }  \Omega,
\end{equation}
where $\theta=c({\cal Q}_h \lambda_0-\lambda_0)$.

{\em Case 1: \ $j=k-1$.} Since ${\cal Q}_h \lambda_0 = \lambda_0\equiv 0$ in this case, then we have $\theta=0$. It follows from $\lambda_0|_{\Gamma_+} = 0$ and the solution uniqueness for \eqref{new:001} that $\lambda_0 \equiv 0$.

{\em Case 2: \ with the $H^{\gamma}$- regularity (\ref{reg}).} In this case, we use the $H^{\gamma}$- regularity assumption (\ref{reg}) and the error estimate for the $L^2$ projection operator $ {\cal Q}_h$ to obtain
$$
\|\lambda_0 \|_{\gamma} \leq C\|\theta\| =C\|c  {\cal Q}_h  \lambda_0-c\lambda_0\| \leq Ch^{\gamma}\|\lambda_0 \|_{\gamma},
$$
which gives
$$
(1-Ch^{\gamma})\|\lambda_0 \|_{\gamma} \leq 0.
$$
This implies that $\lambda_0 \equiv 0$ in $\Omega$ provided that the meshsize $h<h_0$ holds true for a sufficiently small but fixed $h_0>0$ such that $Ch^{\gamma}<1$.

For both cases, from $\lambda_0=\lambda_b$ on each $\partial T$, we obtain  $\lambda_b \equiv  0$ so that $\lambda_h \equiv 0$ is verified.

The proof of $u_h\equiv 0$ in $\Omega$ can be easily carried out by using exactly the same argument for obtaining $u_h\equiv 0$ in Theorem \ref{thmunique1}. Details are thus omitted here. This completes the proof.
\end{proof}

\section{Mass Conservation} \label{Section:MC}

The linear convection equation (\ref{model}) can be rewritten in a conservative form as follows:
\begin{eqnarray}\label{eq1}
\nabla \cdot \textbf{F} +cu&=&f, \\
\label{eq2}
 \textbf{F}&=&\bbeta u.
\end{eqnarray}
On each element $T\in \T_h$, we may integrate (\ref{eq1}) over $T$ to obtain the integral form of the mass conservation:
\begin{equation}\label{mas}
\int_{\partial T}\textbf{F} \cdot \bn ds+\int_T cu dT=\int_T fdT.
\end{equation}

We claim that  the numerical solution arising from the primal-dual weak Galerkin scheme (\ref{32})-(\ref{2}) for the linear convection problem (\ref{model}) retains the local mass conservation property (\ref{mas}) with a numerical flux $\bF_h$. To this end, for any given $T\in {\cal T}_h$, by choosing a test function  $\sigma=\{\sigma_0, \sigma_b=0\}$  in (\ref{32})  such that $\sigma_0=1$ on $T$ and $\sigma_0=0$ elsewhere, we obtain
$$
h_T^{-1} \langle \lambda_0-\lambda_b, 1-0  \rangle_{\partial T}-\tau(\bbeta\cdot\nabla\lambda_0-c\lambda_0,c)_T+ (u_h, \bbeta \cdot \nabla_w \sigma-c\cdot 1 )_T=-(f, 1)_T.
$$
It follows from (\ref{disgradient}) and the usual integration by parts  that
\begin{equation}\label{e1}
\begin{split}
&(f, 1)_T\\
= &h_T^{-1} \langle \lambda_b-\lambda_0, 1  \rangle_{\partial T} + (\nabla \cdot(\bbeta u_h), 1)_T + (cu_h, 1)_T+\tau(\bbeta\cdot\nabla\lambda_0-c\lambda_0,c)_T\\
=&h_T^{-1} \langle \lambda_b-\lambda_0, 1 \rangle_{\partial T} + \langle \bbeta u_h \cdot \bn, 1\rangle_{\partial T}+(c(u_h-\tau c\lambda_0+\tau\bbeta\cdot\nabla\lambda_0), 1)_T\\
=& \langle (h_T^{-1}(\lambda_b-\lambda_0)\bn +\bbeta u_h)\cdot \bn, 1  \rangle_{\partial T} + (c(u_h+\tau (\bbeta\cdot\lambda_0-c\lambda_0)), 1)_T,
\end{split}
\end{equation}
where $\bn$ is the outward normal direction to $\partial T$. The equation (\ref{e1}) implies that the primal-dual weak Galerkin algorithm (\ref{32})-(\ref{2}) conserves mass with a numerical solution and a numerical flux given by
$$
\tilde u_h = u_h + \tau (\bbeta\cdot\nabla\lambda_0- c \lambda_0),\quad \textbf{F}_h|_\pT= \bbeta u_h -h_T^{-1}(\lambda_0-\lambda_b)\bn.
$$

It remains to show that the numerical flux $\bF_h \cdot \bn$ is continuous across each interior edge or flat face. To this end, we choose a test function $\sigma=\{\sigma_0=0, \sigma_b\}$ in (\ref{32}) such that
$\sigma_b$ is arbitrary on one interior edge or flat face $e=\partial T_1\cap \partial T_2$, and $\sigma_b=0$ elsewhere, to obtain
\begin{equation*}\label{mass}
\begin{split}
0=&h_{T_1}^{-1} \langle \lambda_0-\lambda_b, -\sigma_b\rangle_{e\cap\partial T_1}+h_{T_2}^{-1} \langle \lambda_0-\lambda_b, -\sigma_b\rangle_{e\cap\partial T_2}\\
&+(u_h, \bbeta \cdot \nabla_w \sigma)_{T_1\cup T_2} \\
=&h_{T_1}^{-1} \langle \lambda_0-\lambda_b, -\sigma_b\rangle_{e\cap\partial T_1}+h_{T_2}^{-1} \langle \lambda_0-\lambda_b, -\sigma_b\rangle_{e\cap\partial T_2}\\
&+\langle \bbeta u_h \cdot \bn_{T_1}, \sigma_b\rangle_{e\cap\partial T_1}+\langle \bbeta u_h \cdot \bn_{T_2}, \sigma_b\rangle_{e\cap\partial T_2}\\
 =&\langle (\bbeta u_h-h_{T_1}^{-1} (\lambda_0-\lambda_b)\bn_{T_1}  ) \cdot \bn_{T_1},  \sigma_b\rangle_{e\cap\partial T_1}\\
  &+ \langle (\bbeta u_h -h_{T_2}^{-1} (\lambda_0-\lambda_b)\bn_{T_2}  )\cdot \bn_{T_2},  \sigma_b\rangle_{e\cap\partial T_2}\\
=&  \langle \bF_h|_{\pT_1} \cdot \bn_{T_1},  \sigma_b\rangle_{e\cap\partial T_1}+ \langle\bF_h|_{\pT_2} \cdot \bn_{T_2}, \sigma_b\rangle_{e\cap\partial T_2},
\end{split}
\end{equation*}
where we have used (\ref{disgradient}), $\bn_{T_1}$ and $\bn_{T_2}$ are the unit outward normal directions along $e=\partial T_1\cap \partial T_2$ pointing exterior to $T_1$ and $T_2$, respectively.  This shows that
$$
\bF_h|_{\pT_1} \cdot \bn_{T_1} + \bF_h|_{\pT_2} \cdot \bn_{T_2} = 0\qquad \mbox{on } e=\pT_1\cap\pT_2,
$$
 and hence the continuity of the numerical flux along the normal direction on each interior edge or flat face.

The result can be summarized as follows.

\begin{theorem}\label{THM:conservation} Let $(u_h;\lambda_h)$ be the numerical solution of the linear convection model problem \eqref{model} arising from the primal-dual weak Galerkin finite element method (\ref{32})-(\ref{2}). Define a new numerical approximation and a numerical flux function as follows:
\begin{eqnarray*}
\tilde u_h &:=& u_h + \tau (\bbeta\cdot\nabla\lambda_0 - c \lambda_0)\qquad \mbox{in} \ T,  \ T\in \T_h,\\
\textbf{F}_h|_\pT&:=& \bbeta u_h -h_T^{-1}(\lambda_0-\lambda_b)\bn, \quad\mbox{on } \pT, \ T\in \T_h.
\end{eqnarray*}
Then, the flux approximation $\textbf{F}_h$ is continuous across each interior edge or flat face in the normal direction, and the following conservation property is satisfied:
\begin{equation}\label{mas-discrete}
\int_{\partial T}\textbf{F}_h \cdot \bn ds+\int_T c\tilde u_h dT=\int_T fdT.
\end{equation}
\end{theorem}

\section{Error Equations}\label{Section:EE}
Let $u$ and $(u_h; \lambda_h) \in M_{k-1,h}\times W_{j, h}^{0, \Gamma_+}$ be the exact solution of (\ref{model}) and the numerical solution arising from the primal-dual weak Galerkin scheme (\ref{32})-(\ref{2}), respectively. Note that the exact solution  of the Lagrangian multiplier is $\lambda=0$. The error functions for the primal variable $u$ and the dual variable $\lambda$ are thus given by
\begin{align*}
e_h&=u_h-{\cal Q} _hu,
\\
\epsilon_h&=\lambda_h-Q_h\lambda=\lambda_h.
\end{align*}

\begin{lemma}\label{errorequa}
Let $u$ and $(u_h; \lambda_h) \in M_{k-1,h}\times W_{j, h}^{0, \Gamma_+}$ be the exact solution of (\ref{model}) and the numerical solution arising from the primal-dual weak Galerkin scheme (\ref{32})-(\ref{2}), respectively. Then, the error functions $e_h$ and $\epsilon_h$ satisfy the following equations:
\begin{eqnarray}\label{sehv}
s( \epsilon_h , \sigma)+b(e_h, \sigma )&=&\ell_u(\sigma)
,\qquad \forall\sigma\in W^{0, \Gamma_+}_{j,h},\\
b(v, \epsilon_h )&=&0,\qquad\qquad \forall v\in M_{k-1,h}, \label{sehv2}
\end{eqnarray}
where
\begin{equation}\label{lu}
\qquad \ell_u(\sigma) = \sum_{T\in {\cal T}_h}\langle \bbeta (u-{\cal Q}_h u)  \cdot \bn ,\sigma_b-\sigma_0 \rangle_{\partial T} +   ({\cal Q}_hu-u, c\sigma_0 )_T.
\end{equation}
\end{lemma}
\begin{proof}
From (\ref{2}) we have
\begin{align*}
b(v, \epsilon_h) = 0,\qquad\forall v\in M_{k-1,h},
\end{align*}
which gives rise to the equation (\ref{sehv2}).

Next, by subtracting $b( {\cal Q}_hu, \sigma)$ from both sides of (\ref{32}) we arrive at
\begin{equation*}
\begin{split}
& s( \lambda_h-Q_h\lambda, \sigma)+b(u_h-{\cal Q}_hu, \sigma)  \\
=& -(f,\sigma_0)-b(  {\cal Q}_hu, \sigma)+\sum_{e\subset \Gamma_-}\langle \sigma_b, \bbeta \cdot\bn g \rangle_{e}\\
 =&  -(f,\sigma_0) -\sum_{T\in {\cal T}_h}  ({\cal Q}_hu, \bbeta \cdot \nabla_w \sigma-c\sigma_0)_T+\sum_{e\subset \Gamma_-}\langle \sigma_b, \bbeta \cdot\bn g \rangle_{e}\\
                     =&  -(f,\sigma_0)+\sum_{T\in {\cal T}_h}- (  \bbeta{\cal Q}_hu, \nabla \sigma_0)_T+\langle \bbeta {\cal Q}_hu \cdot \bn,  \sigma_0 -\sigma_b\rangle_{\partial T} \\
                     &+ ({\cal Q}_hu, c\sigma_0)_T +\sum_{e\subset \Gamma_-}\langle \sigma_b, \bbeta \cdot\bn g \rangle_{e}\\
                                      =&  -(f,\sigma_0)+ \sum_{T\in {\cal T}_h}-(   \bbeta u, \nabla \sigma_0)_T+\langle \bbeta {\cal Q}_hu \cdot \bn,  \sigma_0 -\sigma_b\rangle_{\partial T} \\
                                      &+ ( {\cal Q}_hu, c\sigma_0)_T +\sum_{e\subset \Gamma_-}\langle \sigma_b, \bbeta \cdot\bn g \rangle_{e}\\
                                      =&  -(f,\sigma_0)+ \sum_{T\in {\cal T}_h}(  \nabla  \cdot ( \bbeta u)+c{\cal Q}_hu, \sigma_0)_T- \langle   \bbeta u  \cdot \bn,  \sigma_0 -\sigma_b\rangle_{\partial T}\\& +\langle \bbeta {\cal Q}_hu \cdot \bn,  \sigma_0 -\sigma_b\rangle_{\partial T}
                                      +\sum_{e\subset \Gamma_-} \{-\langle   \bbeta u  \cdot \bn,   \sigma_b\rangle_{e} +\langle \sigma_b, \bbeta \cdot\bn g \rangle_{e} \}\\
                                      =&\sum_{T\in {\cal T}_h} \langle  \bbeta(u- {\cal Q}_h u) \cdot \bn,  \sigma_b -\sigma_0\rangle_{\partial T} +  ({\cal Q}_hu-u, c\sigma_0 )_T,
  \end{split}
\end{equation*}
where we have used (\ref{disgradient*}), the usual integration by parts, and the facts that $ \nabla  \cdot (\bbeta u)+cu=f$, $u=g$ on $\Gamma_-$, and $\sigma_b=0$ on $\Gamma_+$. This completes the proof of the lemma.
\end{proof}

\section{Some Technical Estimates}\label{Section:TechEst}

Recall that ${\cal T}_h$ is a shape-regular finite element partition of the domain $\Omega$. For any $T\in {\cal T}_h$ and $\phi\in H^1(T)$,
 the following trace inequality holds true \cite{wy3655}:
\begin{equation}\label{tracein}
 \|\phi\|^2_{\partial T} \leq C(h_T^{-1}\|\phi\|_T^2+h_T \|\nabla \phi\|_T^2).
\end{equation}
If $\phi$ is a polynomial on the element $T\in {\cal T}_h$, then from the inverse inequality we have \cite{wy3655},
\begin{equation}\label{trace}
 \|\phi\|^2_{\partial T} \leq Ch_T^{-1}\|\phi\|_T^2.
\end{equation}

The following defines a semi-norm in the finite element space $M_{k-1,h}$:
\begin{equation}\label{mhnorm}
\3bar v\3bar _{M_h}= \Big( \sum_{T\in {\cal T}_h}  h_T^2\| \nabla\cdot (\bbeta v)+cv\|_T^2+   \sum_{e \subset {\cal E}_h\setminus \Gamma_+} h_T\|\ljump \bbeta v\cdot\bn \rjump\|_e^2 \Big)^{\frac{1}{2}}.
\end{equation}

\begin{lemma}\label{mhnormproof} Assume that the solution of the linear transport problem \eqref{model} is unique.
Then, the seminorm  $\3bar \cdot \3bar_{M_h}$  given in (\ref{mhnorm}) defines  a norm in the linear space $M_h$.
\end{lemma}
\begin{proof}
We shall only verify the positivity property for  $\3bar \cdot  \3bar _{M_{h}}$. To this end, we assume that  $\3bar v\3bar_{M_h}=0$ for some $v\in M_{k-1,h}$.  It follows from (\ref{mhnorm}) that $\nabla\cdot ( \bbeta v )+cv=0$ on each $T\in {\cal T}_h$,  $\ljump  \bbeta v \cdot \bn \rjump_e=0$ on each $e\in {\cal E}_h^0$ and $ \bbeta v \cdot \bn=0$ on each $e\subset \Gamma_-$. Therefore, we obtain $\nabla\cdot ( \bbeta v )+cv=0$ in $\Omega$ and $v=0$ on $\Gamma_-$, which gives $v\equiv 0$ in $\Omega$ from the uniqueness assumption for the solution of the linear convection problem \eqref{model}. This completes the proof of the lemma.
\end{proof}

We further introduce a semi-norm in the weak finite element space $W_{j,h}$ as follows. For any $\lambda=\{\lambda_0, \lambda_b\} \in W_{j,h}$, define
\begin{equation}\label{whnorm}
\3bar \lambda \3bar_{W_h}= \Big(\sum_{T\in {\cal T}_h} h_T^{-1} \| \lambda_0-\lambda _b \|_{\partial T}^2+ \tau   \| \bbeta\cdot\nabla\lambda_0-c\lambda_0 \|_T^2\Big)^{\frac{1}{2}},
\end{equation}
with $\tau \ge 0$ being the parameter in the stabilizer $s_T(\cdot,\cdot)$ given in \eqref{stabilizer-local}. It is readily seen that
\begin{equation}\label{triple-bar-one}
\3bar \lambda \3bar_{W_h}^2 = s(\lambda,\lambda).
\end{equation}

\begin{lemma} [inf-sup condition] \label{lem3} For $j=k-1$ and $j=k$, there exists a constant $C>0$, such that for any $v\in M_{k-1,h}$ there exists $\tilde \sigma_v\in W_{j, h}^{0, \Gamma_+}$ satisfying
\begin{equation}\label{inf-sup-new}
|b(v,\tilde\sigma_v)| = \3bar v \3bar_{M_h}^2, \quad \3bar \tilde \sigma_v\3bar_{W_h} \leq C \3bar v \3bar_{M_h}.
\end{equation}
\end{lemma}

\begin{proof}
For any $v\in M_{k-1,h}$ and $\sigma\in W_{j,h}^{0, \Gamma_+}$, from the definition of the weak gradient (\ref{disgradient}) we have
\begin{equation*}
\begin{split}
b(v,\sigma)=&\sum_{T\in {\cal T}_h} (v, \bbeta \cdot \nabla_w \sigma-c\sigma_0)_T\\
=&\sum_{T\in {\cal T}_h}-(\nabla \cdot ( \bbeta  v), \sigma_0)_T+\langle  \bbeta v\cdot \bn, \sigma_b\rangle_{\partial T}-(v, c\sigma_0)_T\\
=&- \sum_{T\in {\cal T}_h} (\nabla \cdot (  \bbeta  v)+cv, \sigma_0)_T+\sum_{e\subset {\cal E}_h\setminus\Gamma_+} \langle \ljump  \bbeta v \cdot \bn\rjump, \sigma_b\rangle _{\partial T},\\
\end{split}
\end{equation*}
where we have used $\sigma_b=0$ on $\Gamma_+$ in the last line. By setting $\tilde\sigma_v=\{\tilde\sigma_0; \tilde\sigma_b \}$ where $\tilde\sigma_0=-h_T^2 (\nabla \cdot (\bbeta  v)+cv)$ on each $T$ and
$ \tilde\sigma_b= h_T \ljump  \bbeta v \cdot \bn\rjump$ on each $e\subset {\cal E}_h\setminus\Gamma_+$, we have
\begin{equation}\label{sig1}
b(v,\tilde\sigma_v)  = \3bar v \3bar_{M_h}^2.
\end{equation}
Now using the triangle inequality and the trace inequality (\ref{trace}) we obtain
\begin{equation}\label{nor1}
\begin{split}
&h_T^{-1} \| \tilde\sigma_0-\tilde\sigma_b \|^2_{\partial T} \\
=& h_T^{-1} \| -h_T^2 (\nabla \cdot (\bbeta  v)+cv)-h_T \ljump  \bbeta v \cdot \bn\rjump\|^2_{\partial T} \\
 \leq & Ch_T^{-1} h_T^{4}h_T^{-1}\|  \nabla \cdot ( \bbeta  v ) +cv\|^2_T+C h_T^{-1}h_T^2\| \ljump  \bbeta v  \cdot \bn\rjump \|^2_{\partial T}\\
 \leq & Ch_T^2 \|  \nabla \cdot (\bbeta  v ) +cv\|^2_T+C h_T \| \ljump  \bbeta v  \cdot \bn\rjump \|^2_{\partial T}.
\end{split}
\end{equation}
Moreover, we have from the inverse inequality that
\begin{equation}\label{nor2}
\begin{split}
\quad \tau  \|\bbeta\cdot\nabla \tilde\sigma_0-c\tilde\sigma_0\|_T^2&\leq 2 \tau \left( \|\bbeta\cdot\nabla \tilde\sigma_0\|_T^2 + \|c\tilde\sigma_0\|_T^2\right)\\
&\leq C \tau \left(h_T^{-2}  \|\tilde\sigma_0\|_T^2 + \|c\tilde\sigma_0\|_T^2\right)\\
& \leq C \tau h_T^{-2} \|h_T^2 (\nabla \cdot (  \bbeta  v)+cv)\|_T^2 \\
&\leq C h^{2}_T \|  \nabla\cdot  (\bbeta  v)+cv \|_T^2.
\end{split}
\end{equation}
It follows from (\ref{nor1}) and (\ref{nor2}) that
\begin{equation}\label{sig}
\3bar \tilde\sigma_v\3bar_{W_h} \leq C \3bar v \3bar_{M_h}.
\end{equation}
Thus, combining (\ref{sig1}) and  (\ref{sig}) gives the inf-sup condition \eqref{inf-sup-new}. This completes the proof of the lemma.
\end{proof}

\section{Error Estimates}\label{Section:ErrorEstimate}
The goal of this section is to derive some error estimates for the solution of the primal-dual weak Galerkin algorithm (\ref{32})-(\ref{2}) by using the error equations (\ref{sehv})-(\ref{sehv2}).

\begin{lemma}\cite{wy3655}
Let ${\cal Q}_h$ be the $L^2$ projection operator onto the finite element space $M_{k-1, h}$ subordinated to the shape-regular finite element partition ${\cal T}_h$ for the domain $\Omega$. For any $0\leq s\leq 1$, $s\leq m \leq k$, $s-1 \leq n \leq k$, there holds
\begin{equation}\label{error2}
\sum_{T\in {\cal T}_h}h_T^{2s}\|u-{\cal Q}_hu\|^2_{s,T}\leq C h^{2m}\|u\|^2_{m},\end{equation}
\begin{equation}\label{errorQ0}
\sum_{T\in {\cal T}_h}h_T^{2s}\|u- Q_0u\|^2_{s,T}\leq C h^{2n+2}\|u\|^2_{n+1}.\end{equation}
\end{lemma}

\begin{theorem} \label{theoestimate} Let $u$ and $(u_h; \lambda_h) \in M_{k-1,h}\times W_{j, h}^{0, \Gamma_+}$ be the exact solution of the linear convection problem (\ref{model}) and its numerical approximation arising from the primal-dual weak Galerkin scheme (\ref{32})-(\ref{2}).
Assume that the solution $u$ of (\ref{model}) is sufficiently regular such that $u\in \oplus_{i=1}^J
H^k(\Omega_i)$ where $\{\Omega_i\}_{i=1}^J$ is a non-overlapping partition of the domain $\Omega$. Then, the following error estimate holds true:
\begin{equation}\label{erres}
\3bar e_h \3bar_{M_h}+\3bar\epsilon_h\3bar_{W_h} \leq
\left\{\begin{split}
Ch^{k}\|u\|_{k},\quad &j=k-1,\\
C(1+\tau^{-\frac12})h^{k}\|u\|_{k},\quad &j=k.
\end{split}\right.
\end{equation}
\end{theorem}

\begin{proof}
By letting $\sigma=\epsilon_h$ in (\ref{sehv}) and $v=e_h$ in (\ref{sehv2}),  we arrive at
$$
s(\epsilon_h, \epsilon_h)= \ell_u(\epsilon_h),
$$
where
$$
\ell_u(\epsilon_h) = \sum_{T\in {\cal T}_h}\langle \bbeta (u-{\cal Q}_h u)  \cdot \bn ,\epsilon_b-\epsilon_0 \rangle_{\partial T} +  ({\cal Q}_hu-u, c\epsilon_0 )_T .
$$
From the relation \eqref{triple-bar-one} and the above two equations we have
\begin{equation}\label{June7:001}
\begin{split}
\3bar\epsilon_h\3bar_{W_h}^2&=s(\epsilon_h,\epsilon_h)\\
 &=|\ell_u(\epsilon_h)| \\
&\leq \sum_{T\in {\cal T}_h} \left| \langle \bbeta (u-{\cal Q}_h u)  \cdot \bn ,\epsilon_b-\epsilon_0 \rangle_{\partial T} \right| + \left| ({\cal Q}_hu-u, c\epsilon_0 )_T  \right|
\end{split}
\end{equation}
It follows from the Cauchy-Schwarz inequality, the trace inequality (\ref{tracein}), (\ref{whnorm}), and (\ref{error2}) with $m=k$ that
\begin{equation}\label{aij}
\begin{split}
 & \sum_{T\in {\cal T}_h}\Big|\langle \bbeta (u-{\cal Q}_h u)  \cdot \bn ,\epsilon_b-\epsilon_0 \rangle_{\partial T} \Big|\\
\leq & \Big(\sum_{T\in {\cal
T}_h}h_T^{-1}\|\epsilon_0-\epsilon_b\|^2_{\partial
T}\Big)^{\frac{1}{2}} \Big(\sum_{T\in {\cal T}_h}h_T\| \bbeta (u-{\cal Q}_h u)  \cdot \bn \|^2_{\partial T}\Big)^{\frac{1}{2}}\\
\leq & C \3bar\epsilon_h \3bar_{W_h} \Big(\sum_{T\in {\cal T}_h} \| \bbeta (u-{\cal Q}_h u)  \|^2_{ T}+h_T^2\|
 \bbeta (u-{\cal Q}_h u) \|^2_{ 1,T}\Big)^{\frac{1}{2}}\\
\leq & C \3bar\epsilon_h \3bar _{W_h} h^{k}\|u\|_{k}.
\end{split}
\end{equation}
Analogously, we have from the standard Cauchy-Schwarz inequality and the estimate (\ref{error2}) with $m=k$  that
\begin{equation}\label{aij-2}
\begin{split}
\sum_{T\in {\cal T}_h}\Big| ({\cal Q}_hu-u, c\epsilon_0 )_T    \Big|  & = \sum_{T\in {\cal T}_h}\Big| ({\cal Q}_hu-u, c\epsilon_0 -\bbeta\cdot\nabla\epsilon_0)_T  \Big| \\
& \leq  \sum_{T\in {\cal T}_h} \|{\cal Q}_hu-u\|_T \|\bbeta\cdot\nabla\epsilon_0 - c\epsilon_0\|_T\\
&\leq  \tau^{-\frac12} \|{\cal Q}_hu-u\| \left(\sum_{T\in {\cal T}_h} \tau \|\bbeta\cdot\nabla\epsilon_0 - c\epsilon_0\|_T^2\right)^{\frac12}\\
& \leq  C  \tau^{-\frac12}  h^{k} \3bar\epsilon_h \3bar _{W_h}\|u\|_{k}.
\end{split}
\end{equation}
Substituting \eqref{aij}-\eqref{aij-2} into \eqref{June7:001} yields
\begin{equation}\label{erreph}
 \3bar\epsilon_h\3bar_{W_h}^2 =s(\epsilon_h,\epsilon_h)=|\ell_u(\epsilon_h)|\leq C(1+\tau^{-\frac12})h^{k} \3bar\epsilon_h \3bar _{W_h}  \|u\|_{k},
\end{equation}
which asserts the estimate for $\epsilon_h$ in \eqref{erres}. Observe that the left-hand side of \eqref{aij-2} is zero for $j=k-1$. Thus, the estimate \eqref{erreph} holds true without the factor $\tau^{-\frac12}$.

To estimate $e_h$, from the inf-sup condition in Lemma \ref{lem3} we have a function $\tilde\sigma\in W_{j,h}^{0, \Gamma_+}$ such that
\begin{equation}\label{inf-sup-app1}
 \3bar e_h \3bar_{M_h}^2 = |b(e_h,\tilde\sigma)|, \quad \3bar \tilde \sigma\3bar_{W_h} \leq C \3bar e_h \3bar_{M_h}.
\end{equation}
It follows from the error equation (\ref{sehv}), the estimates (\ref{erreph}) and (\ref{inf-sup-app1}) that
\begin{equation}\label{errehh}
\begin{split}
\3bar e_h\3bar_{M_h}^2 =  & |b(e_h,\tilde\sigma)| \\
= & |\ell_u(\tilde\sigma)-s(\epsilon_h,\tilde\sigma)| \\
\leq  & |\ell_u(\tilde\sigma)| + |s(\epsilon_h,\tilde\sigma)|\\
\leq &   \left\{\begin{split}
Ch^{k}\|u\|_{k}\3bar   \tilde\sigma\3bar_{W_h} ,\quad &j=k-1\\
C(1+\tau^{-\frac12})h^{k}\|u\|_{k}  \3bar \tilde\sigma\3bar_{W_h},\quad &j=k
\end{split}\right.\\
\leq &  \left\{\begin{split}
Ch^{k}\|u\|_{k}\3bar e_h\3bar_{M_h},\quad &j=k-1,\\
C(1+\tau^{-\frac12})h^{k}\|u\|_{k}\3bar e_h\3bar_{M_h},\quad &j=k,
\end{split}\right.
\end{split}
\end{equation}
which gives rise to the estimate for $e_h$ in \eqref{erres}.
Combining (\ref{erreph}) with  (\ref{errehh}) completes the proof of the theorem.
\end{proof}

 \section{Error Estimate in $L^2$}\label{Section:L2Error}
We shall  use the usual duality argument to establish an error estimate in $L^2$. To this end, consider the problem of seeking $\phi$ such that
\begin{align}\label{dual1}
\bbeta \cdot \nabla \phi - c\phi=&\eta,\qquad \text{in } \Omega,\\
\phi=& 0,\qquad \text{on }\Gamma_+, \label{dual2}
\end{align}
where $\eta\in L^2(\Omega)$. The auxiliary problem  (\ref{dual1})-(\ref{dual2}) is said to satisfy a local $H^1$-regularity if there exists a non-overlapping partition of the domain $\Omega = \bigcup_{i=1}^J \Omega_i$ such that the solution $\phi$ exists, $\phi\in H^1(\Omega_i)$ for all values of $i$, and
\begin{equation}\label{regul}
\left(\sum_{i=1}^J \|\phi\|_{1,\Omega_i}^2\right)^{1/2}\leq C\|\eta\|,
\end{equation}
where $C$ is a generic constant.

Denote by $X_1^*$ the set of all functions $\eta\in L^2(\Omega)$ so that the dual problem (\ref{dual1})-(\ref{dual2}) has the local $H^1$-regularity satisfying the estimate \eqref{regul}.

\begin{theorem}\label{Thm:L2errorestimate}
Let $u_h\in M_{k-1,h}$ be the numerical solution of the linear convection problem (\ref{model}) arising from the PD-WG algorithm (\ref{32})-(\ref{2}), with $\lambda_h\in W_{j, h}^{0, \Gamma_+}$ being the numerical Lagrangian multiplier. Assume that the exact solution $u$ is sufficiently regular such that $u\in \oplus_{i=1}^J H^{k}(\Omega_i)$. Under the local $H^{1}$-regularity assumption (\ref{regul}) for the dual problem (\ref{dual1})-(\ref{dual2}), the following error estimate holds true:
\begin{equation}\label{e0}
\sup_{\eta\in X_1^*,\eta \neq 0}\frac{|(u_h - {\cal Q}_hu, \eta)|}{\|\eta\|} \leq \left\{\begin{split}
Ch^{k}\|u\|_{k},\quad &j=k-1,\\
C(1+\tau^{-\frac12})h^{k}\|u\|_{k},\quad &j=k.
\end{split}\right.
\end{equation}
\end{theorem}

 \begin{proof} For any $\sigma=\{\sigma_0;\sigma_b\}\in W_{j,h}^{0, \Gamma_+}$, we use (\ref{32}), (\ref{model}), and the usual integration by parts to obtain
  \begin{equation*}\label{eq}
  \begin{split}
   & s(\lambda_h, \sigma)+\sum_{T\in {\cal T}_h}(u_h, \bbeta \cdot \nabla_w \sigma- c\sigma_0)_T\\
  =& \sum_{e\subset \Gamma_-} \langle \sigma_b, \bbeta \cdot \bn g\rangle_e-(f, \sigma_0)\\
    =& \sum_{e\subset \Gamma_-} \langle \sigma_b, \bbeta \cdot \bn g\rangle_e-\sum_{T\in {\cal T}_h}(\nabla\cdot(\bbeta u)+cu, \sigma_0)_T\\
  =& \sum_{e\subset \Gamma_-} \langle \sigma_b, \bbeta \cdot \bn g\rangle_e+\sum_{T\in {\cal T}_h} (\bbeta u, \nabla \sigma_0)_T-\langle \bbeta u\cdot \bn, \sigma_0\rangle_{\partial T}-(cu, \sigma_0)_T\\
 =& \sum_{e\subset \Gamma_-} \langle \sigma_b, \bbeta \cdot \bn g\rangle_e+\sum_{T\in {\cal T}_h} (\bbeta {\cal Q}_hu, \nabla \sigma_0)_T-\langle \bbeta u\cdot \bn, \sigma_0\rangle_{\partial T}-(cu, \sigma_0)_T\\
 =& \sum_{e\subset \Gamma_-} \langle \sigma_b, \bbeta \cdot \bn g\rangle_e-\sum_{T\in {\cal T}_h}  (\nabla \cdot(\bbeta {\cal Q}_hu),   \sigma_0)_T-\langle \bbeta {\cal Q}_hu \cdot\bn, \sigma_0\rangle_{\partial T} \\
 &+\langle \bbeta u\cdot \bn, \sigma_0\rangle_{\partial T}+(cu, \sigma_0)_T\\
 =& \sum_{e\subset \Gamma_-} \langle \sigma_b, \bbeta \cdot \bn g\rangle_e-\sum_{T\in {\cal T}_h} (\nabla \cdot(\bbeta {\cal Q}_hu),   \sigma_0)_T-\langle \bbeta {\cal Q}_hu \cdot\bn, \sigma_0-\sigma_b\rangle_{\partial T} \\
 &+\langle \bbeta u\cdot \bn, \sigma_0-\sigma_b\rangle_{\partial T}+(cu, \sigma_0)_T -\langle \bbeta {\cal Q}_hu \cdot\bn,  \sigma_b\rangle_{\partial T}+ \langle \bbeta  u \cdot\bn,   \sigma_b\rangle_{\partial T}  \\
 =& \sum_{e\subset \Gamma_-} \langle \sigma_b, \bbeta \cdot \bn g\rangle_e+\sum_{T\in {\cal T}_h}  \langle \bbeta {\cal Q}_hu \cdot\bn, \sigma_0-\sigma_b\rangle_{\partial T}
-\langle \bbeta u\cdot \bn, \sigma_0-\sigma_b\rangle_{\partial T}\\ &-(cu, \sigma_0)_T+( \bbeta {\cal Q}_hu, \nabla_w  \sigma)_T-\sum_{e\subset \partial \Omega}\langle \bbeta  u \cdot\bn,   \sigma_b\rangle_{e}  \\
 =&  \sum_{T\in {\cal T}_h}  \langle \bbeta ({\cal Q}_hu-u) \cdot\bn, \sigma_0-\sigma_b\rangle_{\partial T} +( {\cal Q}_hu,  \bbeta \cdot \nabla_w  \sigma-c\sigma_0)_T\\
&+ (c ({\cal Q}_hu-u), \sigma_0)_T,\\
\end{split}
\end{equation*}
where we have used $u=g$ on $\Gamma_-$ and $\sigma_b=0$ on $\Gamma_+$ in the last two lines. The last equation can be rewritten as follows
\begin{equation}\label{te1}
\begin{split}
   &  \sum_{T\in {\cal T}_h}(u_h-{\cal Q}_hu, \bbeta \cdot \nabla_w \sigma-c\sigma_0)_T\\
  =& -s(\lambda_h, \sigma)+ \sum_{T\in {\cal T}_h}  \langle \bbeta ({\cal Q}_hu-u) \cdot\bn, \sigma_0-\sigma_b\rangle_{\partial T}+ (c ({\cal Q}_hu-u), \sigma_0)_T. \\
\end{split}
\end{equation}

For any $\eta\in X_1^*$, let $\phi\in \oplus_{i=1}^J H^1(\Omega_i)$ be the solution of the dual problem \eqref{dual1}-\eqref{dual2}.
By letting $\sigma=Q_h \phi=\{Q_0 \phi, Q_b \phi\}$ in (\ref{te1}) we have from  (\ref{dual1}), (\ref{l}), and \eqref{te1} that
  \begin{equation}\label{ett}
  \begin{split}
(u_h-{\cal Q}_hu, \eta) = & (u_h-{\cal Q}_hu, \bbeta \cdot\nabla \phi-c \phi)\\
 =  &  \sum_{T\in {\cal T}_h}(u_h-{\cal Q}_hu, \bbeta \cdot \nabla_w ({Q}_h\phi)-cQ_0 \phi)_T\\
  =&   -s(\lambda_h, Q_h\phi) + \sum_{T\in {\cal T}_h}  \langle \bbeta ({\cal Q}_hu-u) \cdot\bn, Q_0\phi-Q_b\phi\rangle_{\partial T}  \\
& +  (c ({\cal Q}_hu-u), Q_0 \phi)_T\\
  =& I_1+I_2+I_3,\\
\end{split}
\end{equation}
where $\{I_j\}_{j=1}^3$ represent the corresponding terms in the two lines above them. The rest of the proof is devoted to an analysis for the three terms $I_1$, $I_2$, and $I_3$.

For the term $I_1$, we use the Cauchy-Schwarz inequality, the triangle inequality, the trace inequality (\ref{tracein}),  the regularity assumption (\ref{regul}), and the estimate (\ref{errorQ0}) with $n=0$, and \eqref{dual1} to obtain
  \begin{equation} \label{er2}
  \begin{split}
 |I_1|&= |s(\lambda_h, Q_h\phi)| \\
& \leq \sum_{T\in {\cal T}_h} h_T^{-1}  |\langle  \lambda_0-\lambda_b, Q_0\phi-Q_b\phi\rangle_{\partial T}| \\
& \ + \tau |(\bbeta\cdot\nabla\lambda_0- c \lambda_0, \bbeta\cdot\nabla (Q_0\phi) - c Q_0\phi)_T| \\
  \leq &    \Big(\sum_{T\in {\cal T}_h} h_T^{-1}\|  \lambda_0-\lambda_b\|^2_{\partial T}\Big)^{\frac{1}{2}}  \Big(\sum_{T\in {\cal T}_h}  h_T^{-1}  \| Q_0\phi-Q_b\phi\|^2_{\partial T}\Big)^{\frac{1}{2}} \\&
  +\Big(\sum_{T\in {\cal T}_h} \tau   \|\bbeta\cdot\nabla\lambda_0-c\lambda_0\|^2_{T}\Big)^{\frac{1}{2}}  \Big(\sum_{T\in {\cal T}_h} \tau \|\bbeta\cdot\nabla (Q_0\phi)-cQ_0 \phi\|^2_{T}\Big)^{\frac{1}{2}} \\
  \leq & C \3bar \epsilon_h\3bar _{W_h} \Big (\sum_{T\in {\cal T}_h}  h_T^{-2}  \| Q_0\phi- \phi\|^2_{T}+\| Q_0\phi- \phi\|^2_{1, T}+\tau   \|Q_0 \phi\|^2_{1,T}\Big)^{\frac{1}{2}} \\
   \leq & C \3bar \epsilon_h\3bar _{W_h} \Big (\sum_{i=1}^J \|\phi\|_{1,\Omega_i}^2 \Big)^{\frac{1}{2}}\\
      \leq & C \3bar \epsilon_h\3bar _{W_h} \|\eta\|.
\end{split}
\end{equation}

As to the second term $I_2$, from the Cauchy-Schwarz inequality, the trace inequality (\ref{tracein}), the estimate (\ref{error2}) with $m=k$, the estimate \eqref{errorQ0} with $n=0$, and the regularity assumption (\ref{regul}), we obtain
  \begin{equation} \label{er1}
  \begin{split}
 | I_2|= &\left|\sum_{T\in {\cal T}_h}   \langle \bbeta ({\cal Q}_hu-u) \cdot\bn, Q_b\phi-Q_0\phi\rangle_{\partial T}\right|   \\
  \leq & \Big(\sum_{T\in {\cal T}_h}  \| \bbeta ({\cal Q}_hu-u) \cdot\bn\|^2_{\partial T}\Big)^{\frac{1}{2}}  \Big(\sum_{T\in {\cal T}_h}  \| Q_b\phi-Q_0\phi\|^2_{\partial T}\Big)^{\frac{1}{2}} \\
  \leq & \Big(\sum_{T\in {\cal T}_h} h_T^{-1} \|{\cal Q}_hu-u\|^2_{T}+h_T \|{\cal Q}_hu-u \|^2_{1, T}\Big)^{\frac{1}{2}} \\
  & \Big(\sum_{T\in {\cal T}_h} h_T^{-1} \|  \phi-Q_0\phi\|^2_{T}+h_T \|  \phi-Q_0\phi\|^2_{1, T}\Big)^{\frac{1}{2}} \\
  \leq &Ch^k\|u\|_k \sum_{i=1}^J \|\phi\|_{1,\Omega_i}\\
  \leq &Ch^{k}\|u\|_k \|\eta\|.
   \end{split}
\end{equation}

Finally for the term $I_3$, we have
 \begin{equation} \label{er3}
  \begin{split}
 | I_3|= &\left| \sum_{T\in {\cal T}_h}   (c ({\cal Q}_hu-u), Q_0 \phi)_T \right| \\
\leq & C \sum_{T\in {\cal T}_h}  | (c ({\cal Q}_hu-u), Q_0 \phi)_T| \\
\leq & C \|{\cal Q}_hu-u\| \ \|Q_0 \phi\| \\
  \leq &Ch^k\|u\|_k \sum_{i=1}^J \|\phi\|_{0,\Omega_i}\\
  \leq &Ch^{k}\|u\|_k \|\eta\|.
   \end{split}
\end{equation}

Substituting (\ref{er2})-(\ref{er3}) into (\ref{ett}) yields
$$
 |(u_h-{\cal Q}_hu, \eta)| \leq C(h^{k}\|u\|_k + \3bar \epsilon_h\3bar _{W_h}) \|\eta\|,
$$
which, together with \eqref{erres}, gives
\begin{equation*}
\begin{split}
\sup_{\eta \in X_1^*, \eta\neq 0}\frac{|(u_h-{\cal Q}_hu, \eta)| }{\|\eta\|}\leq \left\{\begin{split}
Ch^{k}\|u\|_{k},\quad &j=k-1,\\
C(1+\tau^{-\frac12})h^{k}\|u\|_{k},\quad &j=k.
\end{split}\right.
\end{split}
\end{equation*}
This completes the proof of the Theorem.
\end{proof}

\section{Numerical Experiments}\label{Section:Numerics} In this section, we
shall numerically demonstrate the accuracy and order of convergence for the primal-dual weak Galerkin finite element method (\ref{32})-(\ref{2}) through a series of numerical experiments. The finite element partition $\T_h$ is generated through a successive uniform refinement of a coarse triangulation of the domain by dividing each coarse triangular element into four congruent sub-triangles by connecting the mid-points of the three edges of the triangular element.

The numerical experiments are conducted on both convex and non-convex polygonal domains. The convex domain is given by the unit square $\Omega_1=(0,1)^2$. The non-convex domains include two cases: (i) the L-shaped domain $\Omega_2$ with vertices $A_1=(0, 0)$, $A_2=(2, 0)$, $A_3=(2, 1)$, $A_4=(1, 1)$, $A_5=(1, 2)$, and $A_6=(0, 2)$; (ii) the cracked square domain $\Omega_3=(-1, 1)^2\setminus(0,1)\times{0}$ with a crack along the edge $(0, 1)\times 0$. The inflow boundary $\Gamma_-$ is determined by using the condition of $\bbeta \cdot \bn<0$ with $\bn$ being the unit outward normal direction to $\partial \Omega$. The right-hand side function $f$ and the inflow Dirichlet boundary data $g$ are chosen to match the exact solution $u$ (if the exact solution $u$ is known for the test problem).

The primal-dual weak Galerkin scheme (\ref{32})-(\ref{2}) is implemented for the lowest order $k=1$ and $j=k$. The finite element spaces for the primal variable $u_h$ and the Lagrangian multiplier $\lambda_h$ are specified respectively as follows:
$$
M_{0,h}=\{u_h: \ u_h|_T \in P_0(T),\ \forall T\in {\cal T}_h\},
$$
$$
W_{1,h}=\{\lambda_h=\{\lambda_0,\lambda_b\}: \ \lambda_0\in P_1(T), \lambda_b\in P_1(e), \forall e\subset\pT, \forall T\in {\cal T}_h\}.
$$

Denote by $u_h\in M_{0,h}$ and $\lambda_h=\{\lambda_0, \lambda_b\}\in W_{1,h}$ the numerical solution arising from the PD-WG finite element scheme (\ref{32})-(\ref{2}) with $j=k=1$. The approximation $u_h$ for the primal variable is compared with the exact solution $u$ on each element at the element center - known as the nodal point interpolation $I_h u$. Note that the numerical Lagrangian multiplier $\lambda_h$ approximates the exact solution $\lambda\equiv 0$. The error functions are denoted as
$$
e_h=u_h-I_h u  \quad \text{and}\quad
\epsilon_h=\lambda_h-Q_h\lambda=\{\lambda_0, \lambda_b\}.
$$
The following $L^2$ norms are used to measure the error functions in the numerical tests:
$$
  \|e_h\|  = \Big(\sum_{T\in {\cal
T}_h} \int_T e_h^2 dT\Big)^{\frac{1}{2}},
$$
$$
 \| \lambda_0\|  = \Big(\sum_{T\in {\cal T}_h} \int_T \lambda_0^2 dT\Big)^{\frac{1}{2}}, \qquad  \|\lambda_b\|  = \Big(\sum_{e\in {\cal E}_h}h_T \int_{\partial T}  \lambda_b^2 ds\Big)^{\frac{1}{2}}.
$$

Tables \ref{NE:TRI:Case1-1}-\ref{NE:TRI:Case1-4} illustrate the performance of the PD-WG finite element scheme (\ref{32})-(\ref{2}) for the test problem (\ref{model}) with exact solution $u=1$ on the convex domain $\Omega_1$ and the non-convex domain $\Omega_2$, respectively. The convection vector is given by $\bbeta=[1, -1]$ and the reaction coefficient is $c=1$. The stabilization parameter $\tau$ assumes the values of $\tau=0$ and $\tau=1$, respectively. The numerical results in Tables \ref{NE:TRI:Case1-1}-\ref{NE:TRI:Case1-4} show that the errors are in machine accuracy which is consistent with the theory of $u_h=u$ and $\lambda_h=0$ for this particular test problem. The purpose of this test is to demonstrate the correctness of the code for the PD-WG algorithm (\ref{32})-(\ref{2}) with $j=k=1$.
\begin{table}[H]
\begin{center}
\caption{Numerical error for exact solution $u=1$ on the unit square domain $\Omega_1$: $\bbeta=[1, -1]$, $c=1$, $\tau=1$.}\label{NE:TRI:Case1-1}
\begin{tabular}{|c|c|c|c|c|c|c|}
\hline
$1/h$        & $\|e_h\| $ &    $\| \lambda_0 \| $ &   $\|\lambda_b\|$     \\
\hline
1	&0.00E+00		&7.54E-17	&	4.71E-16
\\
\hline
2	&3.93E-17&		2.58E-16	&	5.11E-16
\\
\hline
4	&8.33E-17	&	7.80E-17	&	1.39E-16
\\
\hline
8	&1.06E-16&		3.85E-17	&	7.00E-17
\\
\hline
16	&1.60E-16	&	5.70E-17	&	8.36E-17
\\
\hline
32&3.45E-16	&	1.99E-16		&2.86E-16
\\
\hline
\end{tabular}
\end{center}
\end{table}

\begin{table}[H]
\begin{center}
\caption{Numerical error for exact solution $u=1$ on the unit square domain $\Omega_1$: $\bbeta=[1, -1]$, $c=1$, $\tau=0$.}\label{NE:TRI:Case1-2}
\begin{tabular}{|c|c|c|c|c|c|c|}
\hline
$1/h$        & $\|e_h\| $ &    $\| \lambda_0 \| $ &   $\|\lambda_b\|$
\\
\hline
1	&0.00E+00	&	5.34E-17	&	7.28E-17
\\
\hline
2	&1.04E-16	&	2.06E-16&		4.94E-16
\\
\hline
4	&1.23E-16		&4.72E-16&		7.54E-16
\\
\hline
8	&1.70E-16&		1.23E-16	&	1.92E-16
\\
\hline
16&	1.81E-16	&	1.35E-16		&1.95E-16
\\
\hline
32	&3.32E-16	&	4.87E-17		&7.04E-17
\\
\hline
\end{tabular}
\end{center}
\end{table}

\begin{table}[H]
\begin{center}
\caption{Numerical error for exact solution $u=1$ on the L-shaped domain $\Omega_2$: $\bbeta=[1, -1]$, $c=1$, $\tau=1$.}\label{NE:TRI:Case1-3}
\begin{tabular}{|c|c|c|c|c|c|c|}
\hline
$1/h$        & $\|e_h\| $ &    $\| \lambda_0 \| $ &   $\|\lambda_b\|$            \\
\hline
1&0		&7.92E-16	&	1.67E-15
\\
\hline
2	&3.64E-16&		6.63E-16	&	1.19E-15
\\
\hline
4&	3.00E-16	&	2.76E-16&		4.34E-16
\\
\hline
8&	4.73E-16	&	7.00E-16&		1.02E-15
\\
\hline
16	&6.84E-16&		8.89E-16	&	1.27E-15					
\\
\hline
32&	1.16E-15	&	5.21E-16	&	7.42E-16
\\
\hline
\end{tabular}
\end{center}
\end{table}
\begin{table}[H]
\begin{center}
\caption{Numerical error for exact solution $u=1$ on the L-shaped domain $\Omega_2$: $\bbeta=[1, -1]$, $c=1$, $\tau=0$.}\label{NE:TRI:Case1-4}
\begin{tabular}{|c|c|c|c|c|c|c|}
\hline
$1/h$        & $\|e_h\| $ &    $\| \lambda_0 \| $ &   $\|\lambda_b\|$
\\
\hline
1&	3.14E-16	&	2.94E-15	&	5.26E-15
\\
\hline
2&	1.88E-16	&	1.37E-15	&	2.35E-15
\\
\hline
4&	2.92E-16	&	1.02E-15	&	1.53E-15
\\
\hline
8	&3.30E-16&		9.58E-16	&	1.40E-15
\\
\hline
16&	8.18E-16	&	1.46E-15&		2.09E-15
\\
\hline
32&1.28E-15	&	1.30E-15	&	1.85E-15
\\
\hline
\end{tabular}
\end{center}
\end{table}

Tables \ref{NE:TRI:Case2-1} - \ref{NE:TRI:Case2-4} illustrate the numerical performance of the PW-WG method when $u=\sin(x)\cos(y)$ is employed as the exact solution in the numerical test. The stabilization parameter is taken to be $\tau=1$ and $\tau=0$. The convection vector and the reaction coefficient are respectively given by $\bbeta=[1, -1]$ and $c=1$. Tables \ref{NE:TRI:Case2-1}- \ref{NE:TRI:Case2-2} show that the convergence for $u_h$ in the discrete $L^2$ norm on the convex domain $\Omega_1$ is at the rate of the optimal order ${\cal O}(h)$ for both $\tau=1$ and $\tau=0$, which is consistent with the theory. Tables \ref{NE:TRI:Case2-3}- \ref{NE:TRI:Case2-4} indicate that the PD-WG method appears to be convergent at a rate slightly higher than ${\cal O}(h)$ on the L-shaped domain $\Omega_2$.
\begin{table}[H]
\begin{center}
\caption{Numerical rates of convergence for exact solution $u=\sin(x)\cos(y)$ on an unit square domain $\Omega_1$: $\bbeta=[1, -1]$, $c=1$, $\tau=1$.}\label{NE:TRI:Case2-1}
\begin{tabular}{|c|c|c|c|c|c|c|}
\hline
$1/h$        & $\|e_h\| $ &  order&  $\| \lambda_0 \| $ &  order&  $\|\lambda_b\|$ &order
\\
\hline
1& 0.06461 	&&	0.6836 	&&0.7941 &	
\\
\hline
2&0.02966 &1.123 	&0.1892 	&1.853 & 0.2238 &	1.827
\\
\hline
4&0.01318 &1.170	&0.04600 &	2.040 &	0.05273	&2.086
\\
\hline
8&0.006401 &1.042 & 1.11E-02	 &2.046 &1.23E-02	&2.098
\\
\hline
16&0.003180 &1.009 &2.73E-03&2.029 &2.94E-03&2.064
\\
\hline
32& 0.001589 &1.001 &6.75E-04&2.016&7.18E-04&2.036
\\
\hline
\end{tabular}
\end{center}
\end{table}

\begin{table}[H]
\begin{center}
\caption{Numerical rates of convergence for exact solution $u=\sin(x)\cos(y)$ on an unit square domain $\Omega_1$: $\bbeta=[1, -1]$, $c=1$, $\tau=0$.}\label{NE:TRI:Case2-2}
\begin{tabular}{|c|c|c|c|c|c|c|}
\hline
$1/h$        & $\|e_h\| $ &  order&  $\| \lambda_0 \| $ &  order&  $\|\lambda_b\|$ &order
\\
\hline
1&0.06022 	&&	0.8152 	&&	0.9571 	&
\\
\hline
2	&0.02685 &	1.165 &	0.2034 &	2.003 &	0.2454 &	1.964
\\
\hline
4&	0.01212&	1.148 &	0.04709&	2.111 &	0.05471 	&2.165
\\
\hline
8& 0.006096 &	0.9914&	1.12E-02&	2.066 &	1.26E-02&	2.124
\\
\hline
16	&0.003092 &	0.9791 &	2.74E-03&	2.036&	2.98E-03	&2.076
\\
\hline
32&	0.001561 &	0.9858 &	6.77E-04	&2.019&	7.23E-04	&2.042
\\
\hline
\end{tabular}
\end{center}
\end{table}

\begin{table}[H]
\begin{center}
\caption{Numerical rates of convergence for exact solution $u=\sin(x)\cos(y)$ on the L-shaped domain $\Omega_2$: $\bbeta=[1, -1]$, $c=1$, $\tau=1$.}\label{NE:TRI:Case2-3}
\begin{tabular}{|c|c|c|c|c|c|c|}
\hline
$1/h$        & $\|e_h\| $ &  order&  $\| \lambda_0 \| $ &  order&  $\|\lambda_b\|$ &order
\\
\hline
1&0.1701 	&&1.114	&&	1.769	&
\\
\hline
2&0.08362 &1.025 &0.3369 &1.725 &0.5205 &1.765
\\
\hline
4&0.03665 &1.190 & 0.1355 &1.314 &0.2020 &1.365
\\
\hline
8&0.01409 &	1.379  &	0.05085 &	1.414 &	0.07398 &	1.449
\\
\hline
16&	0.005947 &	1.245 &	0.01424 &	1.837 &	0.02045	&1.855
\\
\hline
32&0.002693 &1.143&	0.003720&	1.936&0.005309& 1.946
\\
\hline
\end{tabular}
\end{center}
\end{table}

\begin{table}[H]
\begin{center}
\caption{Numerical rates of convergence for exact solution $u=\sin(x)\cos(y)$ on the L-shaped domain $\Omega_2$: $\bbeta=[1, -1]$, $c=1$, $\tau=0$.}\label{NE:TRI:Case2-4}
\begin{tabular}{|c|c|c|c|c|c|c|}
\hline
$1/h$        & $\|e_h\| $ &  order&  $\| \lambda_0 \| $ &  order&  $\|\lambda_b\|$ &order
\\
\hline
1	&0.1590 &&1.364 	&&	2.210 &
\\
\hline
2&0.06908 &	1.203 &	0.4459	&1.614	&0.7237 	&1.611
\\
\hline
4&0.02846 &1.279 &	0.1998	&1.158 &	0.2995 &	1.273
\\
\hline
8 &0.01178  &	1.272 &	0.05795  &	1.786 &	0.08450 &	1.826
\\
\hline
16&0.005410 &	1.123 & 0.01505 &	1.945 &	0.02164 &	1.965
\\
\hline
32&	0.002590&	1.063&	0.003816&	1.979&	0.005449&	1.990
\\
\hline
\end{tabular}
\end{center}
\end{table}

Tables \ref{NE:TRI:Case3-1}-\ref{NE:TRI:Case3-4} show the numerical performance on the unit square domain $\Omega_1$ for various values of the stabilization parameter $\tau=0, 0.001, 1, 1000$. The exact solution is given by $u=\sin(\pi x)\sin(\pi y)$, the convection vector is $\bbeta=[1, 1]$ and the reaction coefficient is $c=-1$. Tables \ref{NE:TRI:Case3-1}-\ref{NE:TRI:Case3-3} demonstrate that the convergence rates for $e_h$ in the discrete $L^2$ norm are a bit higher than the expected optimal order ${\cal O}(h)$ for the case of $\tau=0$, $\tau=0.001$, and $\tau=1$. Observe from Tables \ref{NE:TRI:Case3-1}-\ref{NE:TRI:Case3-2} that the absolute error of $e_h$ is almost the same as the case of $\tau=0$ and $\tau=0.001$ when compared with the results obtained from the same mesh. The numerical results in Table \ref{NE:TRI:Case3-4} indicate that the convergence rate for $e_h$ in the discrete $L^2$ norm is much higher than the optimal order ${\cal O}(h)$ for the stabilization parameter $\tau=1000$, while the absolute error of $e_h$ in this case is larger than the absolute errors of $e_h$ for the case of $\tau=0, 0.001, 1$ compared in the same mesh. For large values of the stabilization parameter $\tau$, the rate of convergence for the Lagrangian multiplier $\lambda_h$ seems to suffer. This test suggests a preference of the PD-WG method with moderate values of the parameter $\tau$.
 \begin{table}[H]
\begin{center}
\caption{Numerical rates of convergence for exact solution $u=\sin(\pi x)\sin(\pi y)$ on an unit square domain $\Omega_1$: $\bbeta=[1, 1]$, $c=-1$, $\tau=0$.}\label{NE:TRI:Case3-1}
\begin{tabular}{|c|c|c|c|c|c|c|}
\hline
$1/h$        & $\|e_h\| $ &  order&  $\| \lambda_0 \| $ &  order&  $\|\lambda_b\|$ &order
\\
\hline
 1	&0.3572 &&	0.9412	&&	1.115&
\\
\hline
2&	0.2116&	0.7556 &	0.9473 &	-0.009340 &1.482 &	-0.4110
\\
\hline
4&	0.09173 &	1.206 & 0.3473 & 1.448 &	0.5146 & 1.526
\\
\hline
8& 0.03086 &	1.571 &	0.1101 &	1.657 &	0.1573 &	1.710
\\
\hline
16	&0.01171 &	1.398 	&0.02954 &	1.899 &0.04127 &	1.930
\\
\hline
32&	0.005356 &	1.128 &	0.007508 &1.976  &0.01037 &1.993
\\
\hline
\end{tabular}
\end{center}
\end{table}
\begin{table}[H]
\begin{center}
\caption{Numerical rates of convergence for exact solution $u=\sin(\pi x)\sin(\pi y)$ on an unit square domain $\Omega_1$: $\bbeta=[1, 1]$, $c=-1$, $\tau=0.001$.}\label{NE:TRI:Case3-2}
\begin{tabular}{|c|c|c|c|c|c|c|}
\hline
$1/h$        & $\|e_h\| $ &  order&  $\| \lambda_0 \| $ &  order&  $\|\lambda_b\|$ &order
\\
\hline
1	&0.3582 	&&	0.9358 &&	1.102 &
\\
\hline
2	&0.2116 &	0.7593 &	0.9461 &	-0.01592 	&1.480	&-0.4250
\\
\hline
4	&0.09169&	1.207&	0.3470 &	1.447 & 0.5142	&1.525
\\
\hline
8	&0.03084&	1.572 &	0.1101&	1.656 &0.1572 &	1.709
\\
\hline
16&	0.01169 &	1.399&	0.02953 &	1.898 &	0.04127 	&1.930
\\
\hline
32&0.005350	&1.128 &0.007508 &1.976 &0.01037	&1.993
\\
\hline
\end{tabular}
\end{center}
\end{table}
\begin{table}[H]
\begin{center}
\caption{Numerical rates of convergence for exact solution $u=\sin(\pi x)\sin(\pi y)$ on an unit square domain $\Omega_1$: $\bbeta=[1, 1]$, $c=-1$, $\tau=1$.}\label{NE:TRI:Case3-3}
\begin{tabular}{|c|c|c|c|c|c|c|}
\hline
$1/h$        & $\|e_h\| $ &  order&  $\| \lambda_0 \| $ &  order&  $\|\lambda_b\|$ &order
\\
\hline
1&0.4448	&&	0.6546 	&&0.2871&
\\
\hline
2&	0.2306 &	0.9474&	0.7959&	-0.2819&	1.216 &	-2.082 \\
\hline
4&	0.09213 &	1.324 &	0.3166 &	1.330&	0.4658 &	1.384
\\
\hline
8&	0.02928 &	1.654&	0.1058&	1.582&	0.1509&	1.627
\\
\hline
16&	0.01022 &	1.518 &	0.02920 &	1.857 &	0.04080&	1.887
\\
\hline
32&	0.004612& 1.149 &	0.007484 &	1.964 &	0.01034&	1.980
\\
\hline
\end{tabular}
\end{center}
\end{table}
\begin{table}[H]
\begin{center}
\caption{Numerical rates of convergence for exact solution $u=\sin(\pi x)\sin(\pi y)$ on an unit square domain $\Omega_1$: $\bbeta=[1, 1]$, $c=-1$, $\tau=1000$.}\label{NE:TRI:Case3-4}
\begin{tabular}{|c|c|c|c|c|c|c|}
\hline
$1/h$        & $\|e_h\| $ &  order&  $\| \lambda_0 \| $ &  order&  $\|\lambda_b\|$ &order
\\
\hline
1&	0.6863 &&	0.1841 &&	0.06168 &
\\
\hline
2&	0.3773 &	0.8632 &	0.2349 &	-0.3514 &	0.2919 &	-2.243
\\
\hline
4&	0.1802 &	1.066 &	0.1257 &	0.9026&	0.1692 	&0.7866
\\
\hline
8&	0.08886 &	1.020 &	0.05015 &	1.325 &	0.06880 &	1.299
\\
\hline
16&	0.03389 &	1.391 &	0.01999 &	1.327 &	0.02760 &	1.318
\\
\hline
32&0.008018 &	2.079 &	0.006678 &	1.581 &	0.009201  &	1.585
\\
\hline
\end{tabular}
\end{center}
\end{table}

In Tables \ref{NE:TRI:Case4-1}-\ref{NE:TRI:Case4-4}, we report some numerical results for the PD-WG method on the L-shaped domain $\Omega_2$ with exact solution $u=\sin(\pi x)\sin(\pi y)$ and various values of $\tau=0, 0.001, 1, 1000$ for the stabilization parameter. The convection vector is given by $\bbeta=[1, 1]$ and the reaction coefficient by $c=1$. Tables \ref{NE:TRI:Case4-1}-\ref{NE:TRI:Case4-3} show that the convergence rate for $e_h$ in the discrete $L^2$ is of ${\cal O}(h)$. Table \ref{NE:TRI:Case4-4} shows that the convergence rate for $e_h$ in the discrete $L^2$ norm is higher than ${\cal O}(h)$ when the stabilization parameter has the value $\tau=1000$. Note that the absolute error of $e_h$ in Table \ref{NE:TRI:Case4-4} with $\tau=1000$ is bigger than those in Tables \ref{NE:TRI:Case4-1}-\ref{NE:TRI:Case4-3} when compared with the result obtained using the same mesh.

\begin{table}[H]
\begin{center}
\caption{Numerical rates of convergence for exact solution $u=\sin(\pi x)\sin(\pi y)$ on the L-shaped domain $\Omega_2$: $\bbeta=[1, 1]$, $c=1$, $\tau=0$.}\label{NE:TRI:Case4-1}
\begin{tabular}{|c|c|c|c|c|c|c|}
\hline
$1/h$        & $\|e_h\| $ &  order&  $\| \lambda_0 \| $ &  order&  $\|\lambda_b\|$ &order
\\
\hline
1	&0.6665 	&&	1.9483 	&&	2.172 &
\\
\hline
2&	0.2840 	&1.231	&2.267 	&-0.2188 &3.394 	&-0.6439
\\
\hline
4	&0.1080 &	1.395 &	0.6346 &	1.837  &0.9082 	&1.902
\\
\hline
8	&0.04060 &1.412 	&0.1547 	&2.036 &0.2158 	&2.074
\\
\hline
16	&0.01904 &	1.092 	&0.03706 &	2.062	&0.05090&	2.084	
\\
\hline
32	&	0.009381 	&	1.021 &	0.009125	&	2.022 	&0.01244	&	2.033
\\
\hline
\end{tabular}
\end{center}
\end{table}
\begin{table}[H]
\begin{center}
\caption{Numerical rates of convergence for exact solution $u=\sin(\pi x)\sin(\pi y)$ on the L-shaped domain $\Omega_2$: $\bbeta=[1, 1]$, $c=1$; $\tau=0.001$.}\label{NE:TRI:Case4-2}
\begin{tabular}{|c|c|c|c|c|c|c|}
\hline
$1/h$        & $\|e_h\| $ &  order&  $\| \lambda_0 \| $ &  order&  $\|\lambda_b\|$ &order
\\
\hline
1	&0.66831 	&&	1.9405 	&&2.160 &
\\
\hline
2&	0.2840 	&1.235	&2.265 &-0.2230	&3.390&	-0.6501
\\
\hline
4&	0.1079 	&1.396	&0.6342	&1.836 &0.9076	&1.901
\\
\hline
8&	0.04055 &	1.412 &	0.1547 &	2.036 &	0.2157 &2.073
\\
\hline
16&0.01902 	&1.092 &	0.03706 &	2.061 &	0.05089 	&2.084
\\
\hline
32&0.009374&	1.021&	0.009125&	2.022 &	0.01243&2.033
\\
\hline
\end{tabular}
\end{center}
\end{table}
\begin{table}[H]
\begin{center}
\caption{Numerical rates of convergence for exact solution $u=\sin(\pi x)\sin(\pi y)$ on the L-shaped domain $\Omega_2$: $\bbeta=[1, 1]$, $c=1$; $\tau=1$.}\label{NE:TRI:Case4-3}
\begin{tabular}{|c|c|c|c|c|c|c|}
\hline
$1/h$        & $\|e_h\| $ &  order&  $\| \lambda_0 \| $ &  order&  $\|\lambda_b\|$ &order
\\
\hline
1&0.9362 	&&	1.076	&&0.5230 	&
\\
\hline
2&	0.3222  &	1.539	&1.774	&-0.7216 &2.616 &	-2.322
\\
\hline
4	&0.1133 &	1.507 &	0.5773 	&1.619 	&0.8230 	&1.668
\\
\hline
8&	0.03770	&1.588 	&0.1515 & 1.930 	&0.2111 	&1.963
\\
\hline
16&	0.01745 &	1.111 	&0.03689 	&2.038	&0.05067	&2.059
\\
\hline
32&	0.008610& 1.019&0.009107&	2.018&	0.01242&	2.029
\\
\hline
\end{tabular}
\end{center}
\end{table}
\begin{table}[H]
\begin{center}
\caption{Numerical rates of convergence for exact solution $u=\sin(\pi x)\sin(\pi y)$ on the L-shaped domain $\Omega_2$: $\bbeta=[1, 1]$, $c=1$; $\tau=1000$.}\label{NE:TRI:Case4-4}
\begin{tabular}{|c|c|c|c|c|c|c|}
\hline
$1/h$        & $\|e_h\| $ &  order&  $\| \lambda_0 \| $ &  order&  $\|\lambda_b\|$ &order
\\
\hline
1&	1.208 	&&	0.2749 	&&	0.1019 &
\\
\hline
2&	0.5886 &	1.037 &	0.3499 &	-0.3481 &	0.3809 &	-1.903
\\
\hline
4&	0.2668 &	1.142&	0.1865 &	0.9081&	0.2370 &	0.6846
\\
\hline
8&	0.1119&	1.254 &	0.07380&	1.337 &	0.1001 &	1.243
\\
\hline
16&0.03311 &	1.756 &	0.02724 &	1.438 &	0.03748 &	1.417
\\
\hline
32&	0.009289 &	1.834&	0.008355&1.705&	0.01143&	1.714
\\
\hline
\end{tabular}
\end{center}
\end{table}

Figure \ref{ux} illustrates the numerical performance of the PD-WG method for a test problem with the following configuration: the domain is the unit square $\Omega_1$, the exact solution is $u=\sin(x)\cos(y)$, the convection vector is $\bbeta=[y-0.5, -x+0.5]$, the reaction coefficient $c=1$, and the stabilizer parameters $\tau=0, 1, 10000$. Figure \ref{ux} shows that the convergence rate for $e_h$ in the discrete $L^2$ norm is of order ${\cal O}(h^{0.9})$ which is a bit lower than the expected optimal order ${\cal O}(h)$ when $\tau=0$ (left figure) and $\tau=1$ (middle figure) are employed. We conjecture that the slight deterioration on the convergence rate is caused by the rotational nature of the flow. The right figure in Figure \ref{ux} indicates that the convergence rate for $e_h$ is of ${\cal O}(h^{1.3})$ when $\tau=10000$, which is better than the theoretical prediction.

\begin{figure}[h]
\centering
\begin{tabular}{cc}
\resizebox{1.5in}{1.5in}{\includegraphics{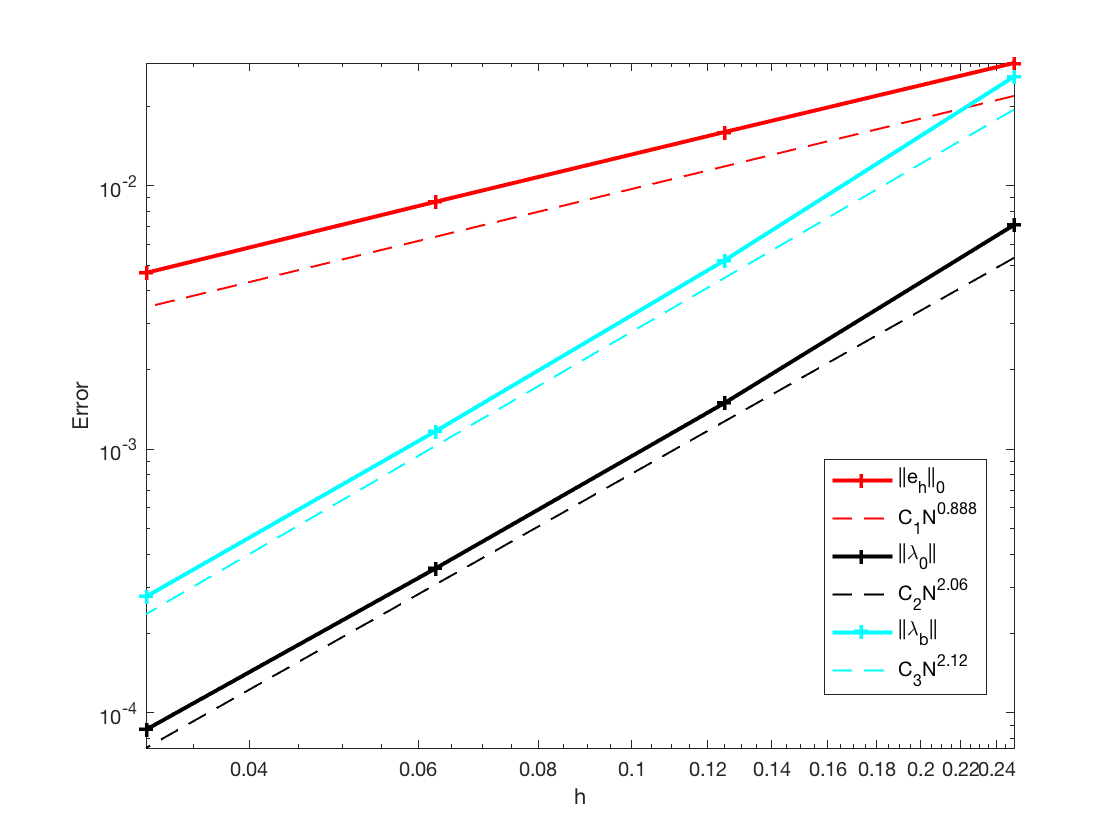}}
\resizebox{1.5in}{1.5in}{\includegraphics{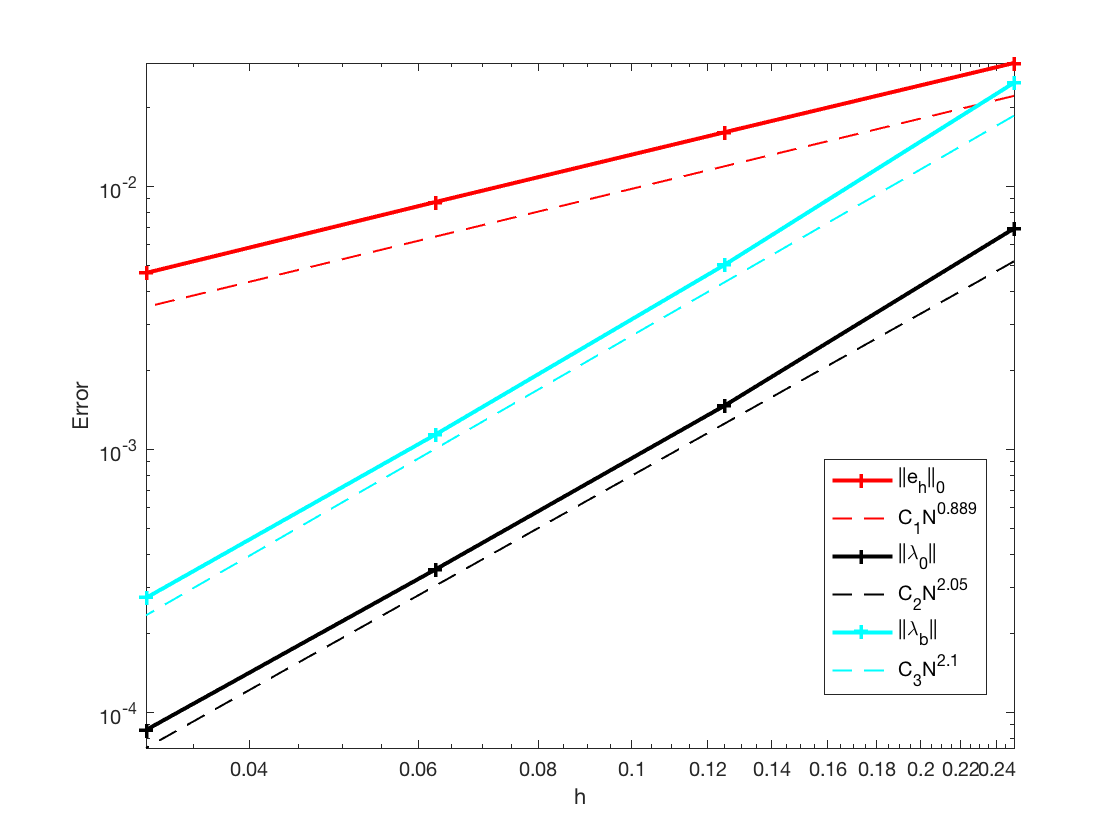}}
\resizebox{1.5in}{1.5in}{\includegraphics{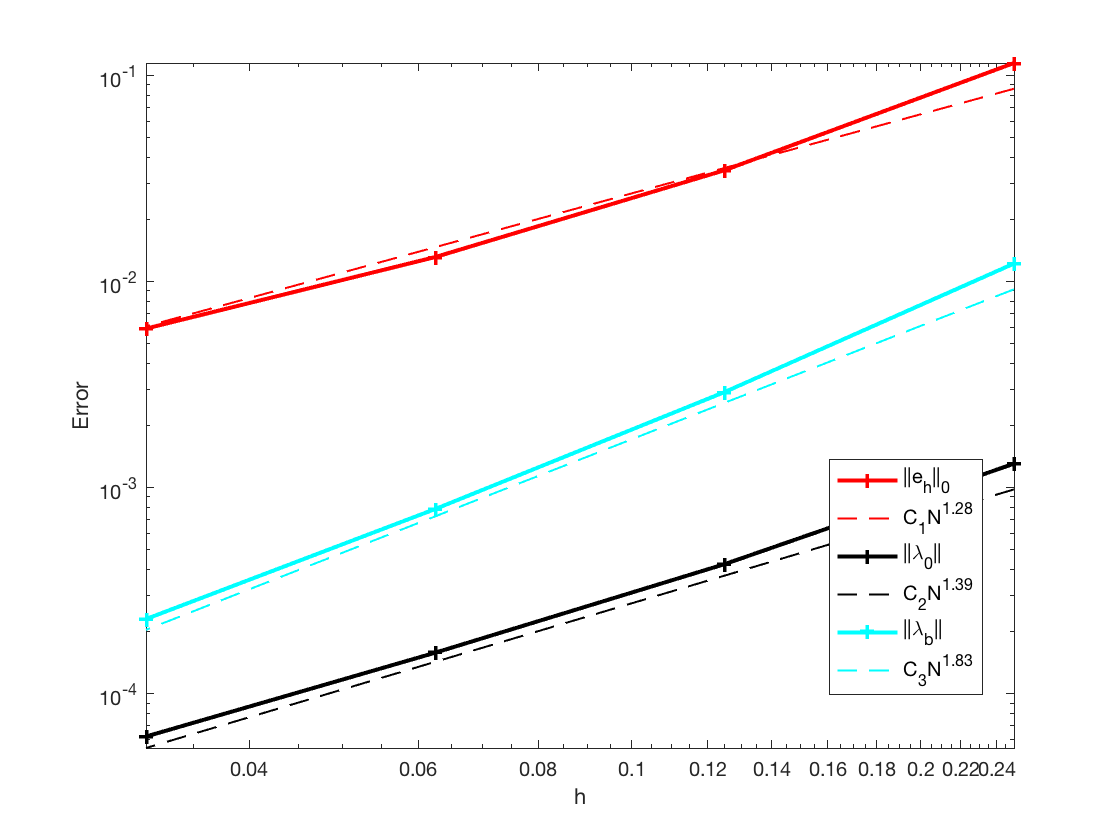}}
\end{tabular}
\caption{Numerical rates of convergence for exact solution $u=\sin(x)\cos(y)$ on an unit square domain $\Omega_1$: $\bbeta=[y-0.5, -x+0.5]$ and $c=1$, $\tau=0$ (left), $\tau=1$ (middle), $\tau=10000$ (right).}
\label{ux}
 \end{figure}

Tables \ref{NE:TRI:Case10-1}-\ref{NE:TRI:Case10-2} show the numerical results on the L-shaped domain $\Omega_2$ for the stabilization parameters $\tau=1$ and $\tau=0$. The exact solution is given by $u=\sin(x)\cos(y)$. The convection vector is given by $\bbeta=[y-1, -x+1]$ and the reaction coefficient is chosen as $c=1$. Tables \ref{NE:TRI:Case10-1}-\ref{NE:TRI:Case10-2} indicate that the convergence rate for $e_h$ in the discrete $L^2$ norm is of ${\cal O}(h^{0.9})$. This slight deterioration of the convergence rate might be caused by the rotational nature of the flow.

\begin{table}[H]
\begin{center}
\caption{Numerical rates of convergence for exact solution $u=\sin(x)\cos(y)$ on the L-shaped domain $\Omega_2$: $\bbeta=[y-1, -x+1]$, $c=1$, $\tau=1$.}\label{NE:TRI:Case10-1}
\begin{tabular}{|c|c|c|c|c|c|c|}
\hline
$1/h$        & $\|e_h\| $ &  order&  $\| \lambda_0 \| $ &  order&  $\|\lambda_b\|$ &order
\\
\hline
1& 0.2577 	&&1.829	&&	3.933 &
\\
\hline
2	&0.1345 	&0.9376 &	0.4039 &	2.179 	&0.7490 &	2.393
\\
\hline
4&	0.06931&	0.9570 &	0.07443 &	2.440&	0.1299&	2.528
\\
\hline
8&	0.03677 &	0.9143 &	0.01586 &	2.231&	0.02558  &	2.344
\\
\hline
16&	0.01970	&0.9006 &	3.71E-03	&2.096 	&0.005621&	2.186
\\
\hline
32&	0.01050 & 0.9072 & 9.03E-04&2.038 &	0.001318& 2.093
\\
\hline

\end{tabular}
\end{center}
\end{table}
\begin{table}[H]
\begin{center}
\caption{Numerical rates of convergence for exact solution $u=\sin(x)\cos(y)$ on the L-shaped domain $\Omega_2$: $\bbeta=[y-1, -x+1]$, $c=1$, $\tau=0$.}\label{NE:TRI:Case10-2}
\begin{tabular}{|c|c|c|c|c|c|c|}
\hline
$1/h$        & $\|e_h\| $ &  order&  $\| \lambda_0 \| $ &  order&  $\|\lambda_b\|$ &order
\\
\hline
1&	0.1779	&&3.173 &&	6.112 	&
\\
\hline
2&	0.1147 &	0.6330 &	0.4859 &	2.707 &	0.8792 	&2.797
\\
\hline
4&	0.06512 &	0.8168&	0.07896 &	2.622&	0.1396&	2.655
\\
\hline
8&	0.03560 &	0.8714 &	0.01633 &	2.273 &	0.02684&	2.378
\\
\hline
16&0.01929 &	0.8843 &	0.003771 &	2.115 &	0.005791 &	2.212
\\
\hline
32&0.01033 &	0.9002&	9.13E-04	&2.047&	0.001341 & 2.111
\\
\hline
\end{tabular}
\end{center}
\end{table}

In Figure \ref{ux1}, we present some numerical results on the cracked square domain $\Omega_3$. The exact solution is chosen to be $u=\sin(x)\sin(y)$. The convection vector is $\bbeta=[y, -x]$ and the reaction coefficient is $c=1$.  Figure \ref{ux1} indicates that the convergence rate for $e_h$ in the discrete $L^2$ norm arrives at an order of ${\cal O}(h^{1.1})$ when the stabilization parameters $\tau=1$ and $\tau=0$ are employed.

 \begin{figure}[h]
\centering
\begin{tabular}{cc}
\resizebox{2.4in}{2.2in}{\includegraphics{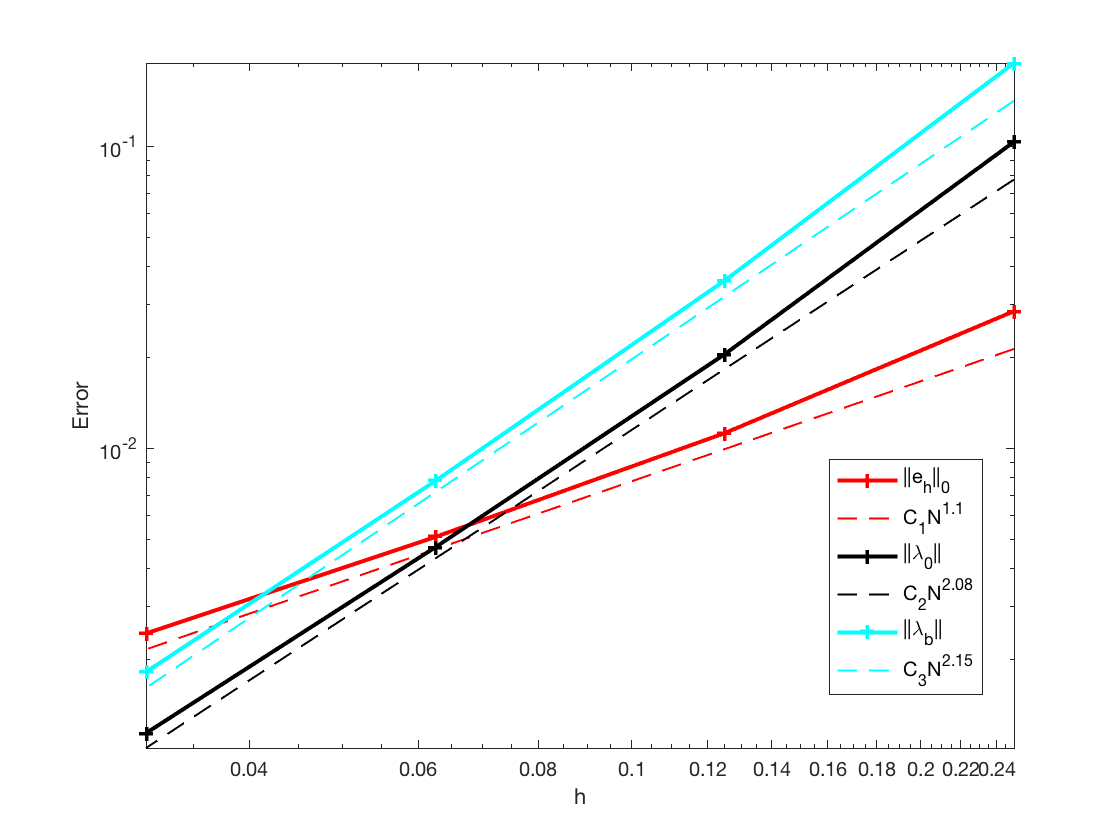}}
\resizebox{2.4in}{2.2in}{\includegraphics{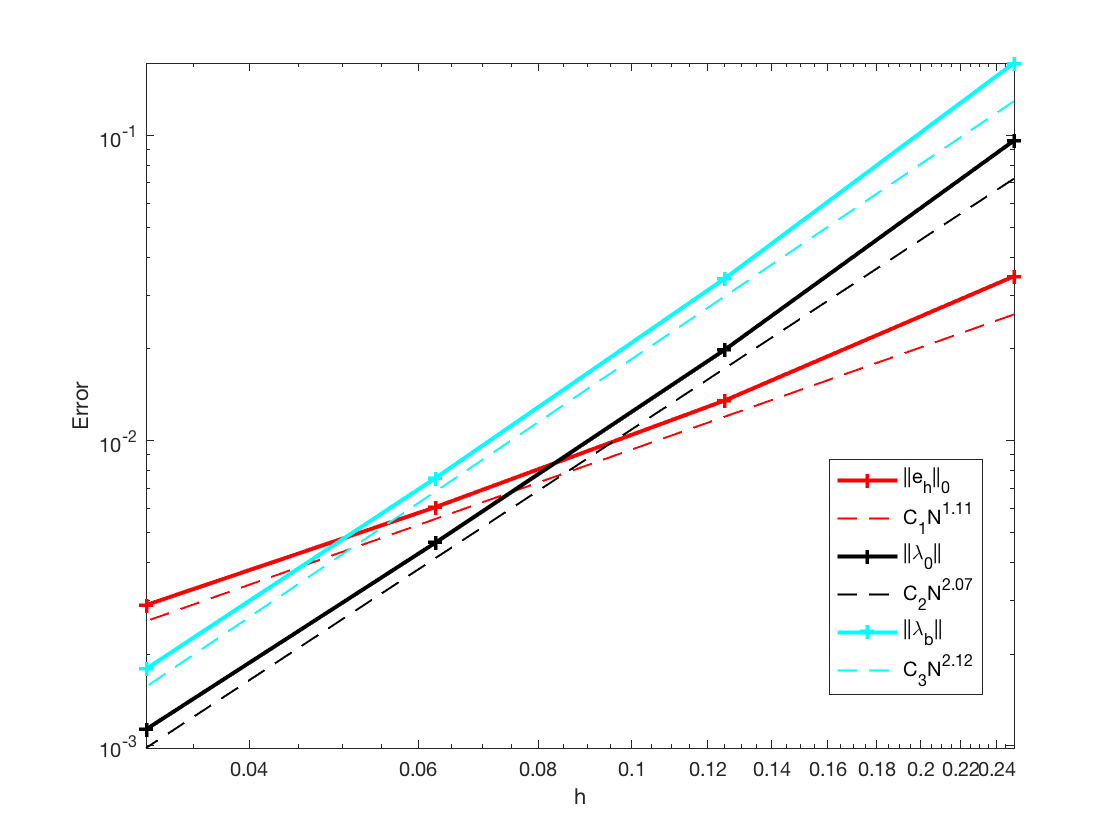}}
\end{tabular}
\caption{Numerical rates of convergence for exact solution $u=\sin(x)\sin(y)$ on the cracked square domain $\Omega_3$: $\bbeta=[y, -x]$, $c=1$, $\tau=0$ (left), $\tau=1$ (right).}
\label{ux1}
 \end{figure}

Tables \ref{NE:TRI:Case11-1}-\ref{NE:TRI:Case11-2} illustrate the numerical results on the cracked square domain $\Omega_3$ when the convection vector is given by $\bbeta=[y, -x]$ and the reaction coefficient by $c=1$. The exact solution is given by $u=\sin(\pi x)\cos(\pi y)$. The numerical results in Tables \ref{NE:TRI:Case11-1}-\ref{NE:TRI:Case11-2} show that the convergence rate for $e_h$ in the discrete $L^2$ norm is of ${\cal O}(h)$ when $\tau=1$ and $\tau=0$. 
 \begin{table}[H]
\begin{center}
\caption{Numerical rates of convergence for exact solution $u=\sin(\pi x)\cos(\pi y)$ on the cracked square domain $\Omega_3$: $\bbeta=[y, -x]$, $c=1$, $\tau=1$.}\label{NE:TRI:Case11-1}
\begin{tabular}{|c|c|c|c|c|c|c|}
\hline
$1/h$        & $\|e_h\| $ &  order&  $\| \lambda_0 \| $ &  order&  $\|\lambda_b\|$ &order
\\
\hline
1&0.5008 &&	2.009 &&	3.088 &
\\
\hline
2&	0.2900 &	0.7881 &	1.0270 &	0.9680 &	1.615 &	0.9350
\\
\hline
4&	0.1555&	0.8994 &	0.3126 &	1.716 &	0.4735 &1.770
\\
\hline
8&	0.07209 &	1.109 &	0.06884 &	2.183 &	0.09902 &	2.257
\\
\hline
16&	0.03540 &	1.026 &	0.01601&	2.105&	0.02202 &	2.169
\\
\hline
32&	0.01765&	1.004&	0.003888&	2.041&0.005219&	2.077
\\
\hline
\end{tabular}
\end{center}
\end{table}
\begin{table}[H]
\begin{center}
\caption{Numerical rates of convergence for exact solution $u=\sin(\pi x)\cos(\pi y)$ on the cracked square domain $\Omega_3$: $\bbeta=[y, -x]$, $c=1$, $\tau=0$.}\label{NE:TRI:Case11-2}
\begin{tabular}{|c|c|c|c|c|c|c|}
\hline
$1/h$        & $\|e_h\| $ &  order&  $\| \lambda_0 \| $ &  order&  $\|\lambda_b\|$ &order
\\
\hline
1&	0.3958 &&	3.348 && 4.985 &
\\
\hline
2&	0.2562 &	0.6275 &	1.417 &	1.240 &	2.207 &	1.176
\\
\hline
4&	0.1477 &	0.7943 &	0.3559 &	1.994 &	0.5442 &	2.020
\\
\hline
8&	0.06988 &	1.080 &	0.07175 &	2.310 &	0.1042&	2.385
\\
\hline
16&0.03472 &	1.009 &	0.01628 &	2.140 &	0.02250 &	2.211
\\
\hline
32&	0.01740&	0.9972&	0.003922&	2.053&	0.005271&	2.094
\\
\hline
\end{tabular}
\end{center}
\end{table}

Figure \ref{ux2} illustrates the numerical performance of the PD-WG method on the unit square domain $\Omega_1$ for the exact solution $u=\sin(x)\cos(y)$. The convection vector $\bbeta(x, y)$ is piece-wisely defined in the sense that $\bbeta(x, y)=[y, -x]$ for $y<1-x$ and  $\bbeta(x, y)=[y-1, 1-x]$ otherwise; and the reaction coefficient is $c=1$. The stabilization parameter assumes the values of $\tau=1$ and $\tau=0$. The numerical results in Figure \ref{ux2} show the convergence rate for $e_h$ in the discrete $L^2$ norm arrives at an optimal order of ${\cal O}(h)$, which is in good consistency with the theory.

\begin{figure}[h]
\centering
\begin{tabular}{cc}
\resizebox{2.4in}{2.2in}{\includegraphics{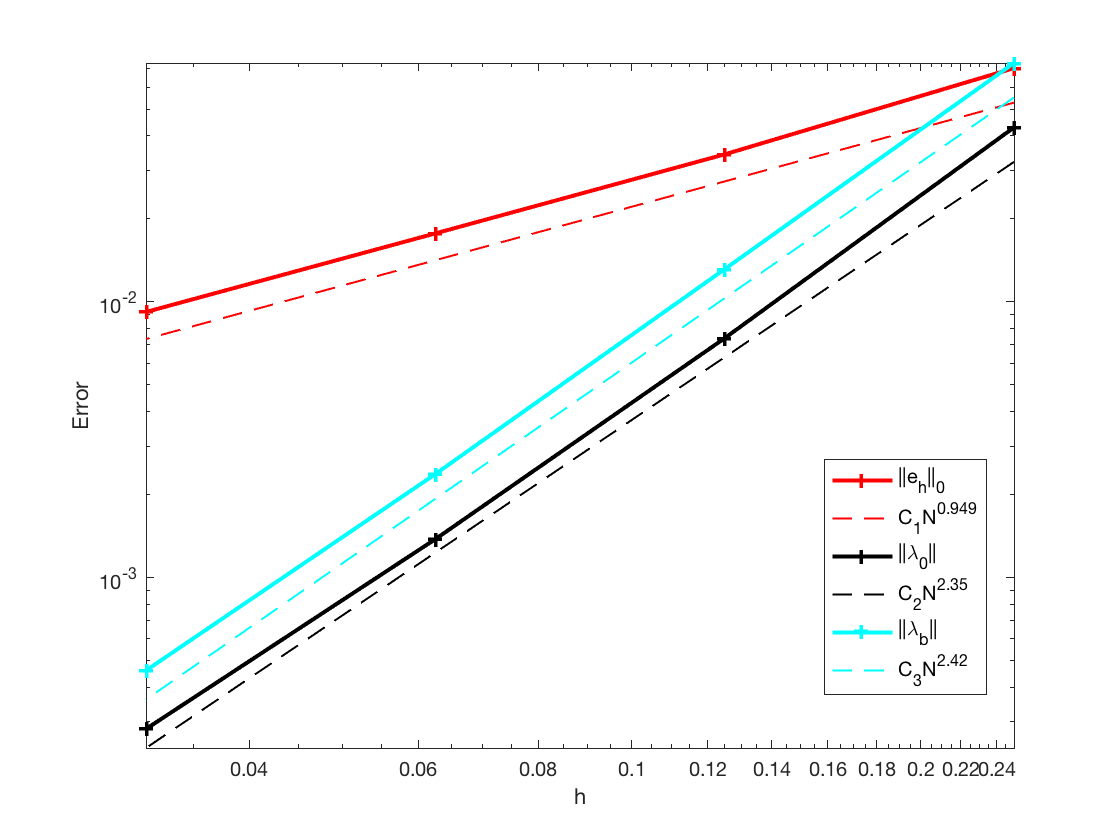}}
\resizebox{2.4in}{2.2in}{\includegraphics{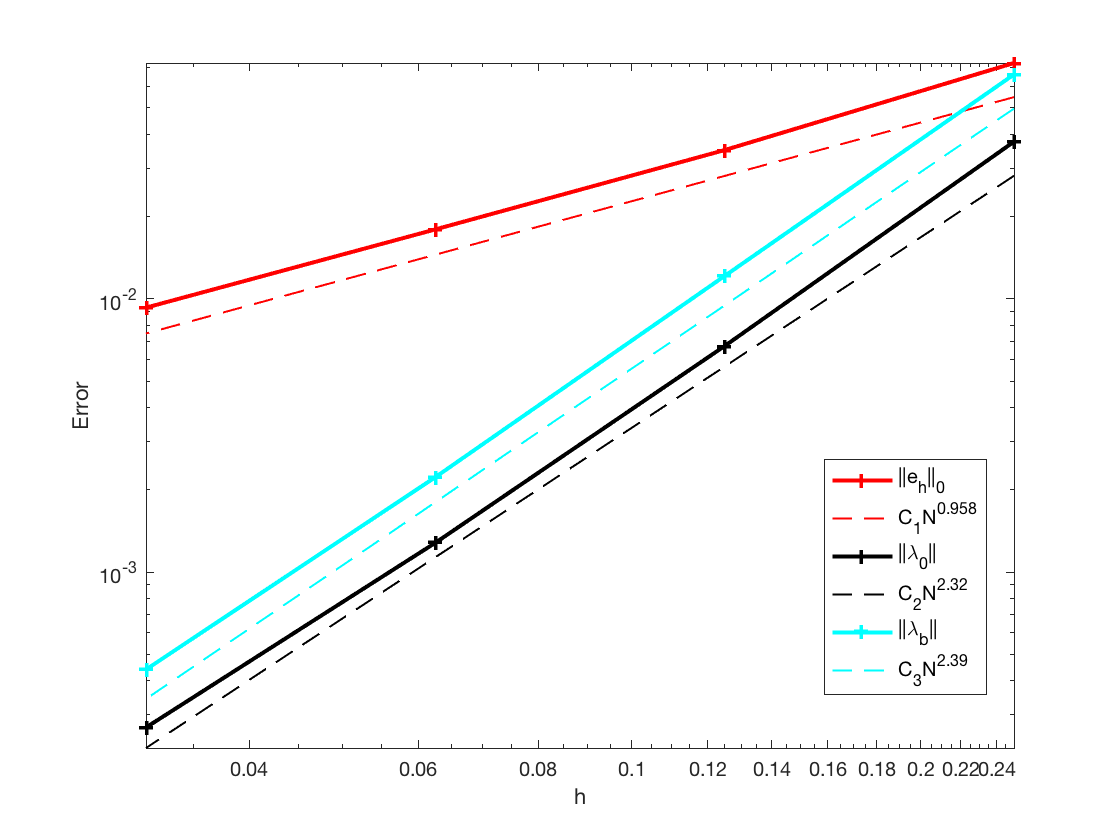}}
\end{tabular}
\caption{Numerical rates of convergence for exact solution $u=\sin(x)\cos(y)$ on an unit square domain $\Omega_1$: $\bbeta(x,  y)=[y, -x]$ for $y<1-x$ and $\bbeta=[y-1, 1-x]$ otherwise, $c=1$, $\tau=0$ (left), $\tau=1$ (right).}
\label{ux2}
 \end{figure}

Tables \ref{NE:TRI:Case8-1}-\ref{NE:TRI:Case8-4} illustrate the performance of the PD-WG method for the model problem \eqref{model} with piecewise constant coefficients. The exact solution is chosen as $u=\sin(\pi x)\cos(\pi y)$, and the domain is given by either the unit square $\Omega_1$ or the L-shaped domain $\Omega_2$. The convection vector  $\bbeta(x, y)$ is piece-wisely defined as follows: $\bbeta(x, y)=[1, -1]$ if $y<1-x$ and $\bbeta(x, y)=[-1, 1]$ elsewhere. The reaction coefficient is given by $c=1$. The stabilization parameter has values $\tau=1$ and $\tau=0$. Tables \ref{NE:TRI:Case8-1}-\ref{NE:TRI:Case8-2} show that the convergence rate for $e_h$ in the discrete $L^2$ norm on the unit square domain $\Omega_1$ is of ${\cal O}(h)$, which is consistent with the theory. The numerical results in Tables \ref{NE:TRI:Case8-3}-\ref{NE:TRI:Case8-4} suggest that the convergence for $e_h$ in the discrete $L^2$ norm on the L-shaped domain $\Omega_2$ is of order ${\cal O}(h)$.

\begin{table}[H]
\begin{center}
\caption{Numerical rates of convergence for exact solution $u=\sin(\pi x)\cos(\pi y)$ on an unit square domain $\Omega_1$: $\bbeta(x, y)=[1, -1]$ if$y<1-x$ and $\bbeta(x, y)=[-1, 1]$ otherwise, $c=1$, $\tau=1$.}\label{NE:TRI:Case8-1}
\begin{tabular}{|c|c|c|c|c|c|c|}
\hline
$1/h$        & $\|e_h\| $ &  order&  $\| \lambda_0 \| $ &  order&  $\|\lambda_b\|$ &order
\\
\hline
1&	0.1411&&	0.6130 &&	0.7236	&
\\
\hline
2&	0.1427 &	-0.01700&	0.4904 &	0.3219 &	0.6096 	&0.2472
\\
\hline
4&	0.07178 &	0.9916 &	0.1188&	2.046 &	0.1388 	&2.134
\\
\hline
8&	0.03058 &	1.231 &	3.33E-02&	1.833 &	0.03937 &	1.818
\\
\hline
16&	0.01446 &	1.080 &	8.44E-03	&1.981 	&9.86E-03&	1.997
\\
\hline
32&	0.007131&	1.020 	&2.11E-03&	2.003 &	2.43E-03	&2.018
\\
\hline
\end{tabular}
\end{center}
\end{table}

\begin{table}[H]
\begin{center}
\caption{Numerical rates of convergence for exact solution $u=\sin(\pi x)\cos(\pi y)$ on an unit square domain $\Omega_1$: $\bbeta(x, y)=[1, -1]$ if$y<1-x$ and $\bbeta(x, y)=[-1, 1]$ otherwise, $c=1$, $\tau=0$.}\label{NE:TRI:Case8-2}
\begin{tabular}{|c|c|c|c|c|c|c|}
\hline
$1/h$        & $\|e_h\| $ &  order&  $\| \lambda_0 \| $ &  order&  $\|\lambda_b\|$ &order
\\
\hline
1&0.1953 &&	0.7864 	&&1.058 &
\\
\hline
2&	0.1136 &	0.7815 &	0.5297 &	0.5703 &0.6675 	&0.6650
\\
\hline
4&	0.06042 &	0.9109 &	0.1236 &	2.099 &	0.1466 &	2.187
\\
\hline
8&	0.02824 &	1.097 &	0.03433 &	1.849 &	0.04123&	1.830
\\
\hline
16&	0.01383 &	1.030 &	0.008656 &	1.988&	0.01026&	2.007
\\
\hline
32&0.006899 &	1.003&	0.002155 &	2.006 &	0.002524 &	2.023
\\
\hline
\end{tabular}
\end{center}
\end{table}

\begin{table}[H]
\begin{center}
\caption{Numerical rates of convergence for exact solution $u=\sin(\pi x)\cos(\pi y)$ on the L-shaped domain $\Omega_2$: $\bbeta(x, y)=[1, -1]$ if $y<1-x$ and $\bbeta(x, y)=[-1, 1]$ otherwise; $c=1$; $\tau=1$.}\label{NE:TRI:Case8-3}
\begin{tabular}{|c|c|c|c|c|c|c|}
\hline
$1/h$        & $\|e_h\| $ &  order&  $\| \lambda_0 \| $ &  order&  $\|\lambda_b\|$ &order
\\
\hline
1&	0.4095 	&&2.047 	&&	2.881&
\\
\hline
2&	0.2575 &	0.6693	&1.484 	&0.4640 &	2.180 	&0.4020
\\
\hline
4&	0.1138 &	1.178 &	0.5139 &	1.530 &	0.7252&	1.588
\\
\hline
8&	0.04564 &	1.319 &	0.1377 &	1.900&	0.1912 &	1.923
\\
\hline
16&	0.02152&	1.085 &	0.03337 &	2.045 &	0.04585&	2.060
\\
\hline
32& 0.01060 & 1.022 &  0.008210& 2.0231& 0.01122& 2.031
\\
\hline
\end{tabular}
\end{center}
\end{table}

\begin{table}[H]
\begin{center}
\caption{Numerical rates of convergence for exact solution $u=\sin(\pi x)\cos(\pi y)$ on the L-shaped domain $\Omega_2$: $\bbeta(x, y)=[1, -1]$ if $y<1-x$ and $\bbeta(x, y)=[-1, 1]$ otherwise, $c=1$, $\tau=0$.}\label{NE:TRI:Case8-4}
\begin{tabular}{|c|c|c|c|c|c|c|}
\hline
$1/h$        & $\|e_h\| $ &  order&  $\| \lambda_0 \| $ &  order&  $\|\lambda_b\|$ &order
\\
\hline
1	&0.3716 	&&	3.493 	&&	5.315	&
\\
\hline
2&	0.2102 &	0.8219 &	1.907 &	0.8731 &	2.835 	&0.9066
\\
\hline
4&	0.1047 &	1.006 &	0.5779 &	1.723 &	0.8193 &	1.791
\\
\hline
8&	0.04487  &	1.222&	0.1414&	2.031 &	0.1967 &	2.059
\\
\hline
16&	0.02160 &	1.055&	0.03373 &	2.067 &	0.04639 &	2.084
\\
\hline
32	&0.01070&	1.013&	0.008287&	2.025&	0.01133 &	2.034
\\
\hline
\end{tabular}
\end{center}
\end{table}


Fig. \ref{ux3} illustrates the numerical performance of the PD-WG on the unit square domain $\Omega_1$. The  convection vector is given by $\bbeta=(\cos(\frac{\pi}{6}), \sin(\frac{\pi}{6}))=(b_1, b_2)$ and reaction coefficient is $c=0$. The exact solution is chosen as $u=\left((y-\frac{b_2}{b_1}x-\frac{1}{2})^2+\frac{1}{10}\right)^{-1}$. The stabilization parameter assumes the values of $\tau=1$ and $\tau=0$. The numerical results in Fig. \ref{ux3} suggest a convergence of $e_h$ in the discrete $L^2$ norm at the rate of ${\cal O}(h^{1.3})$, which outperforms the expected optimal order of ${\cal O}(h)$.

\begin{figure}[h]
\centering
\begin{tabular}{cc}
\resizebox{2.4in}{2.2in}{\includegraphics{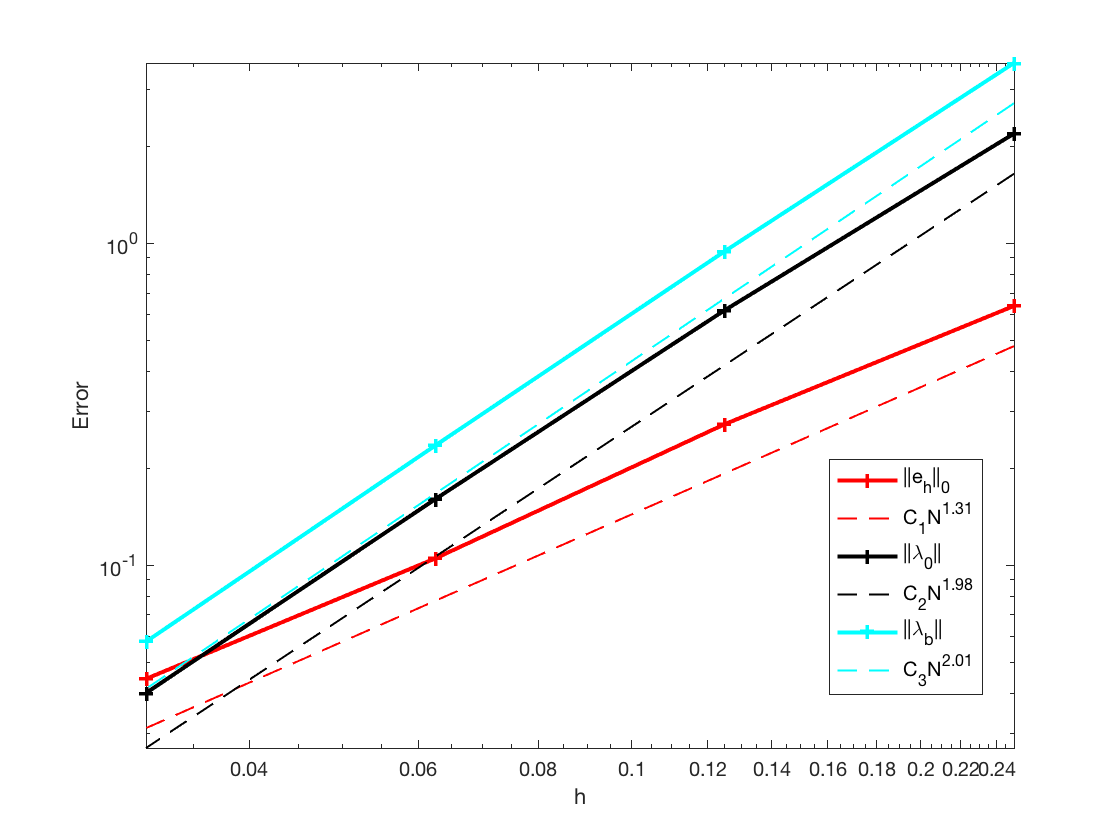}}
\resizebox{2.4in}{2.2in}{\includegraphics{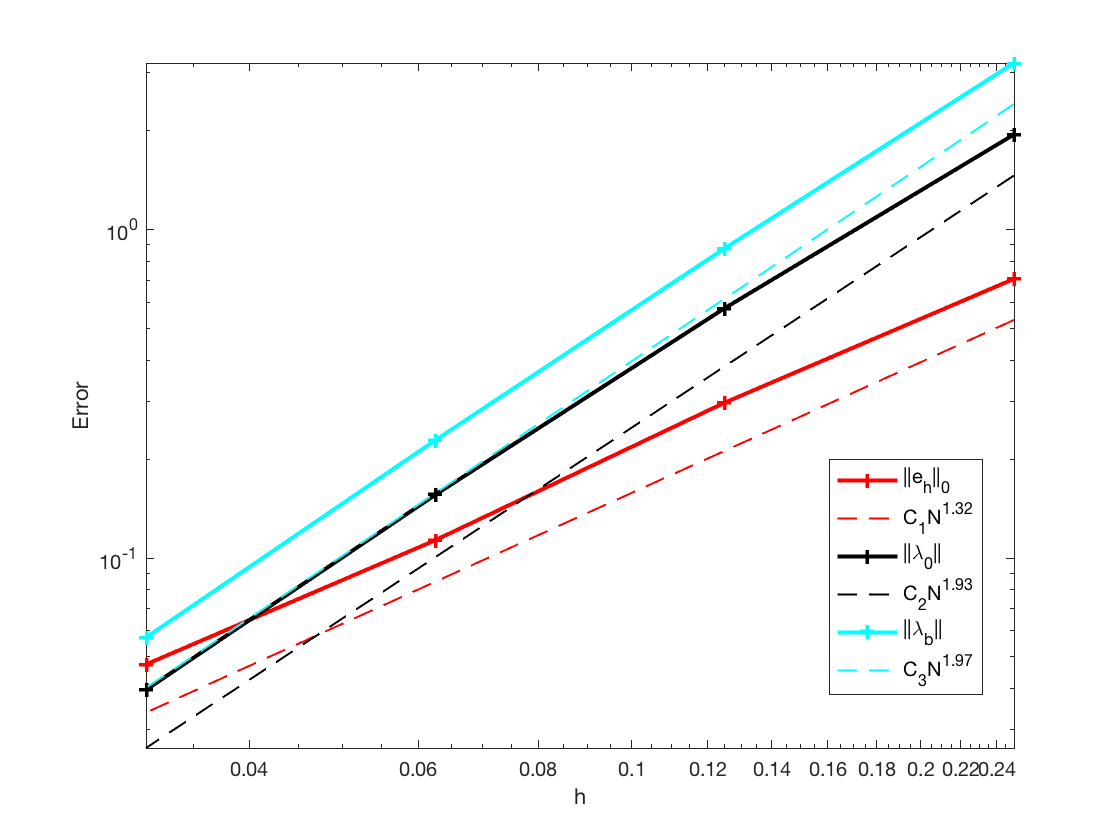}}
\end{tabular}
\caption{Numerical rates of convergence for exact solution $u=\left((y-\frac{b_2}{b_1}x-\frac{1}{2})^2+\frac{1}{10}\right)^{-1}$ on the unit square domain $\Omega_1$: $\bbeta=(\cos(\frac{\pi}{6}), \sin(\frac{\pi}{6}))=(b_1, b_2)$, $c=0$, $\tau=0$ (left), $\tau=1$ (right).}
\label{ux3}
\end{figure}

Figure \ref{ux4} shows the numerical results on the unit square domain $\Omega_1$. The convection vector is given by $\bbeta=(\cos(\frac{\pi}{6}), \sin(\frac{\pi}{6}))=(b_1, b_2)$ and the reaction coefficient is given by $c=0$. The exact solution is
\begin{equation}\label{NE:mytest}
u= \left\{
\begin{array}{cc}
  \left((y-\frac{b_2}{b_1}x-\frac{1}{2})^2+\frac{1}{10}\right)^{-1},\quad & y\geq \frac{b_2}{b_1}x;   \\
\frac{20}{7}, \quad & y<\frac{b_2}{b_1}x.   \\
\end{array}
\right.
\end{equation}
The stabilization parameter assumes the values of $\tau=1$ and $\tau=0$. 
The numerical results in Fig. \ref{ux4} suggest a convergence rate of order ${\cal O}(h^{1.3})$ for $e_h$ in the discrete $L^2$ norm, which is better than the theoretical result of ${\cal O}(h)$.

\begin{figure}[h]
\centering
\begin{tabular}{cc}
\resizebox{2.4in}{2.2in}{\includegraphics{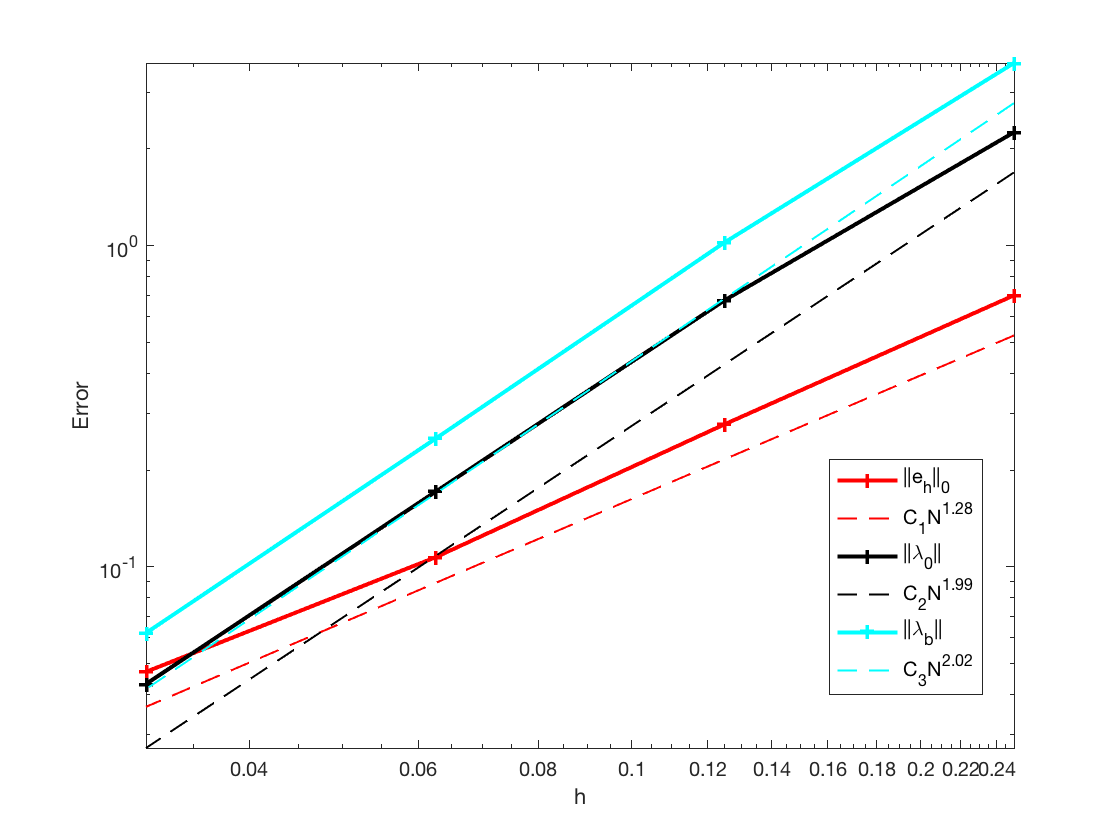}}
\resizebox{2.4in}{2.2in}{\includegraphics{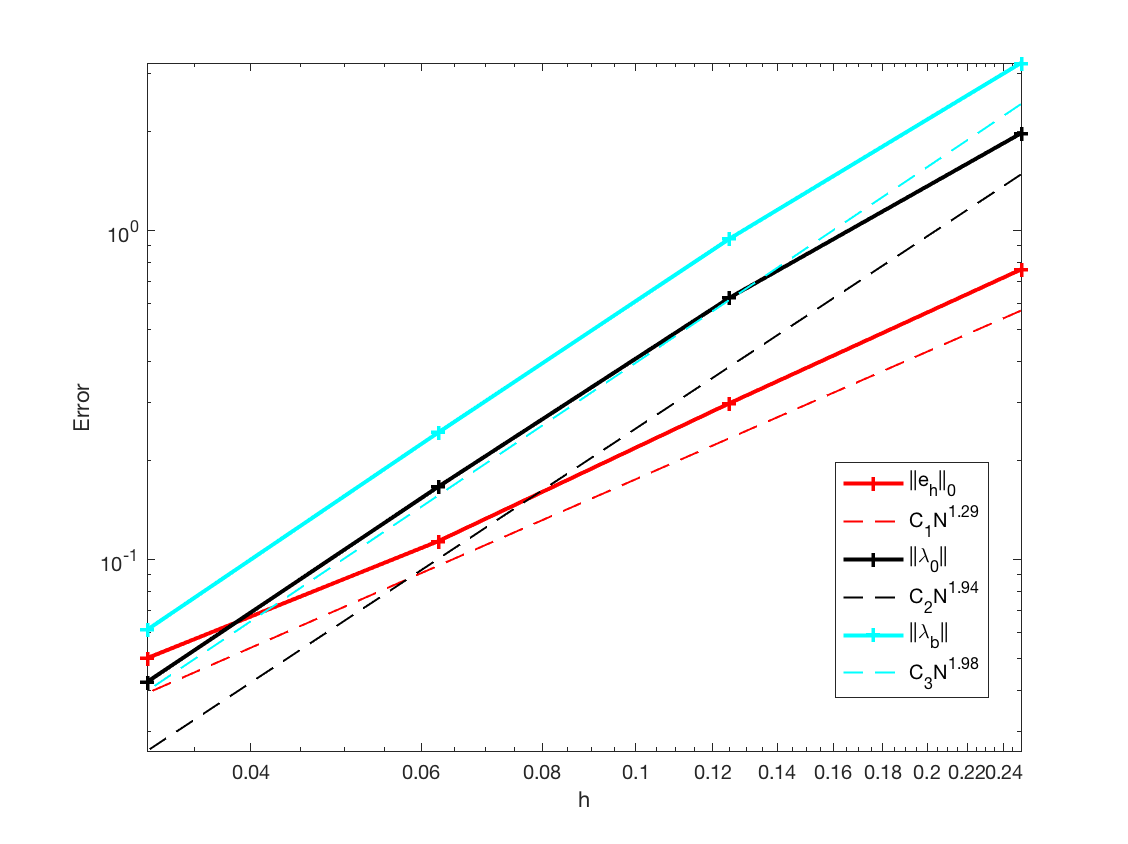}}
\end{tabular}
\caption{Numerical rates of convergence on an unit square domain $\Omega_1$ with discontinuous solution given by \eqref{NE:mytest}: $\bbeta=[\cos(\frac{\pi}{6}), \sin(\frac{\pi}{6})]=(b_1, b_2)$, $c=0$, $\tau=0$ (left), $\tau=1$ (right).}
\label{ux4}
\end{figure}

\bigskip

The rest of the numerical tests shall assume no explicit formulation on the exact solution for the linear transport problem (\ref{model}). The plot of the numerical solution $u_h$ arising from the PD-WG scheme (\ref{32})-(\ref{2}) will be shown for each numerical experiment. To produce a smooth plot, we occasionally apply a post-processing technique to generate a post-processed approximation for the plotting purpose. The post-processed approximation has values on the vertices and the midpoint of each edge for each triangular element $T\in {\cal T}_h$. The value of the post-processed approximation at each vertex point is calculated as the simple average of $u_h$ on all the elements sharing the same vertex. Similarly, the value of the post-processed approximation at the midpoint of each edge is computed as the simple average of $u_h$ on the elements sharing the same edge.

Figures \ref{u=1-1}-\ref{u=1-2} show the plots of the numerical solution $u_h$ arising from the PD-WG scheme (\ref{32})-(\ref{2}) on the unit square domain $\Omega_1$ with and without using the post-processing technique. The configuration of the test problem is as follows: the convection vector is $\bbeta=[1, -1]$; the reaction coefficient is $c=0$; the stabilization parameter is $\tau=0$; the load function is $f=0$; the inflow boundary data is given by $g=1$ on the inflow boundary edge $\{0\}\times(0,1)$ and by $g=-1$ on the inflow boundary edge $(0,1)\times\{1\}$. The exact solution is known to be $u=1$ for $x<1-y$ and $u=-1$ otherwise. The left ones in Figures \ref{u=1-1}-\ref{u=1-2} are the surface plots for the numerical solution and the right ones are the contour plots of the numerical solution. Figures \ref{u=1-1}-\ref{u=1-2} show that the numerical solution $u_h$ obtained from the PD-WG scheme (\ref{32})-(\ref{2}) is consistent with the exact solution $u$ of the linear transport problem (\ref{model}).

\begin{figure}[h]
\centering
\begin{tabular}{cc}
\resizebox{2.4in}{2.1in}{\includegraphics{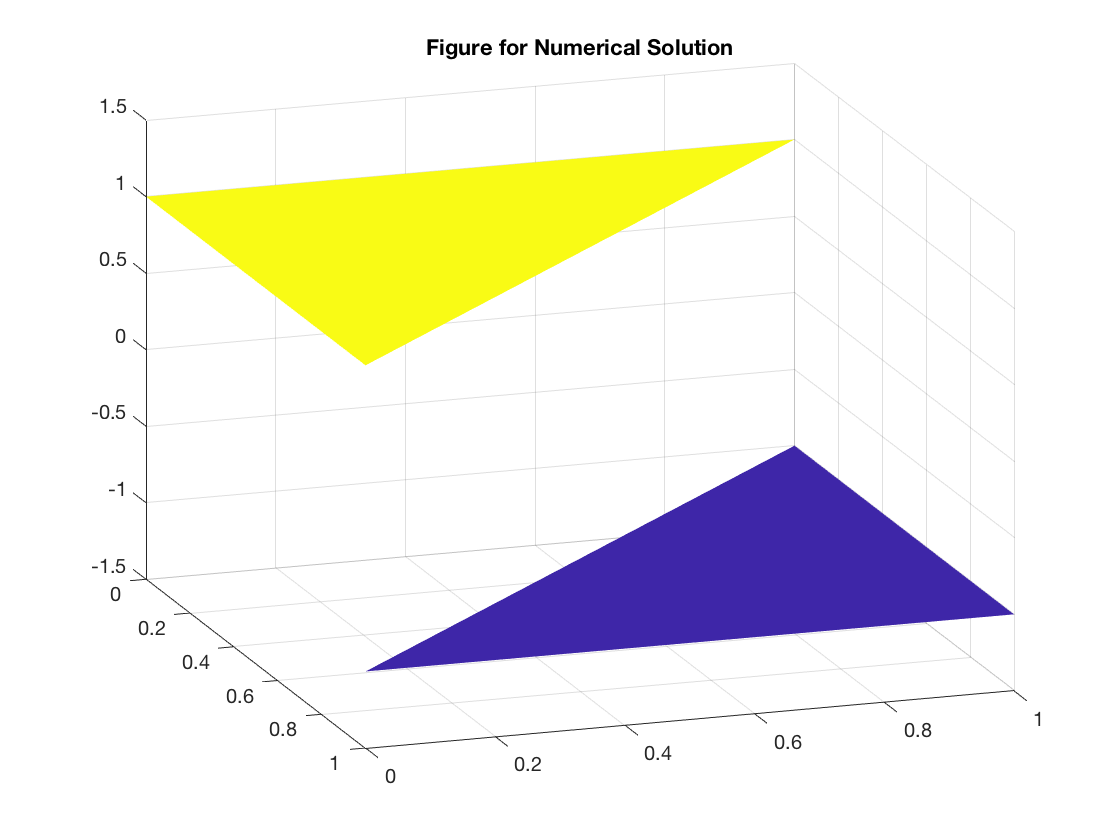}}
\resizebox{2.4in}{2.1in}{\includegraphics{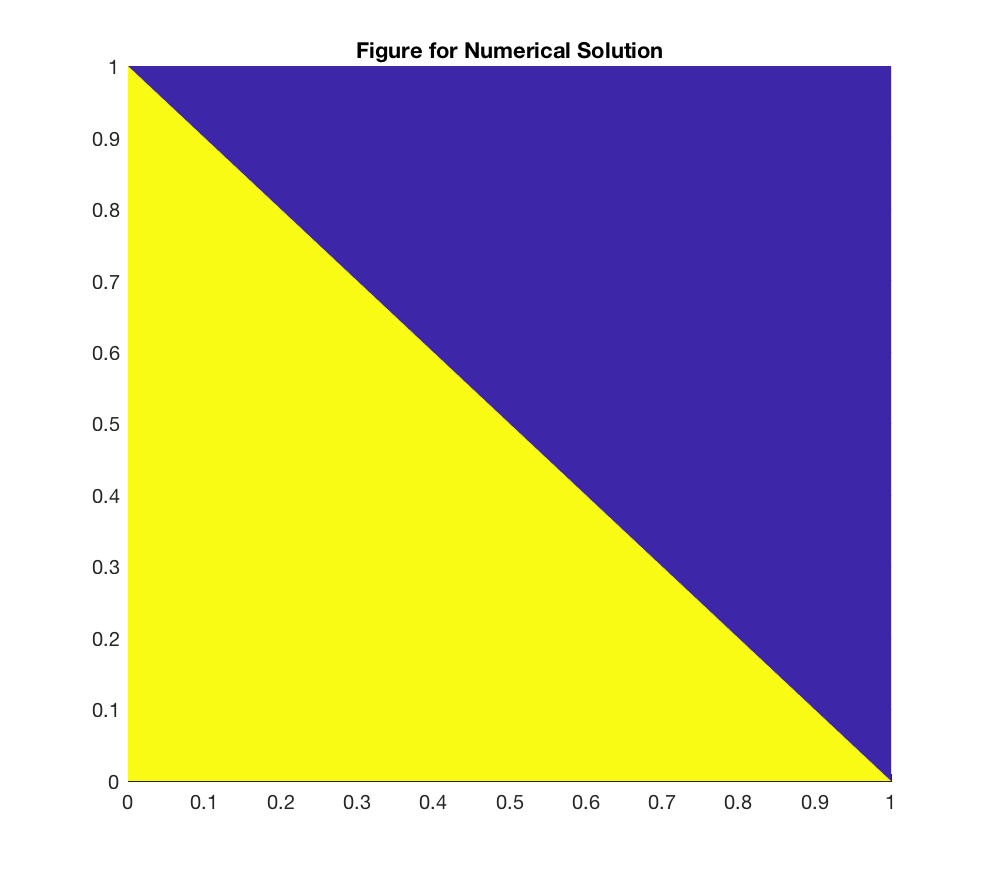}}
\end{tabular}
\caption{(without post-processing) Plots of the numerical solution $u_h$ on the unit square domain $\Omega_1$: $\bbeta=[1, -1]$, $c=0$, $\tau=0$, and discontinuous Dirichlet data at inflow boundary. Surface plot (left); contour plot (right).}
\label{u=1-1}
\end{figure}

\begin{figure}[h]
\centering
\begin{tabular}{cc}
\resizebox{2.4in}{2.1in}{\includegraphics{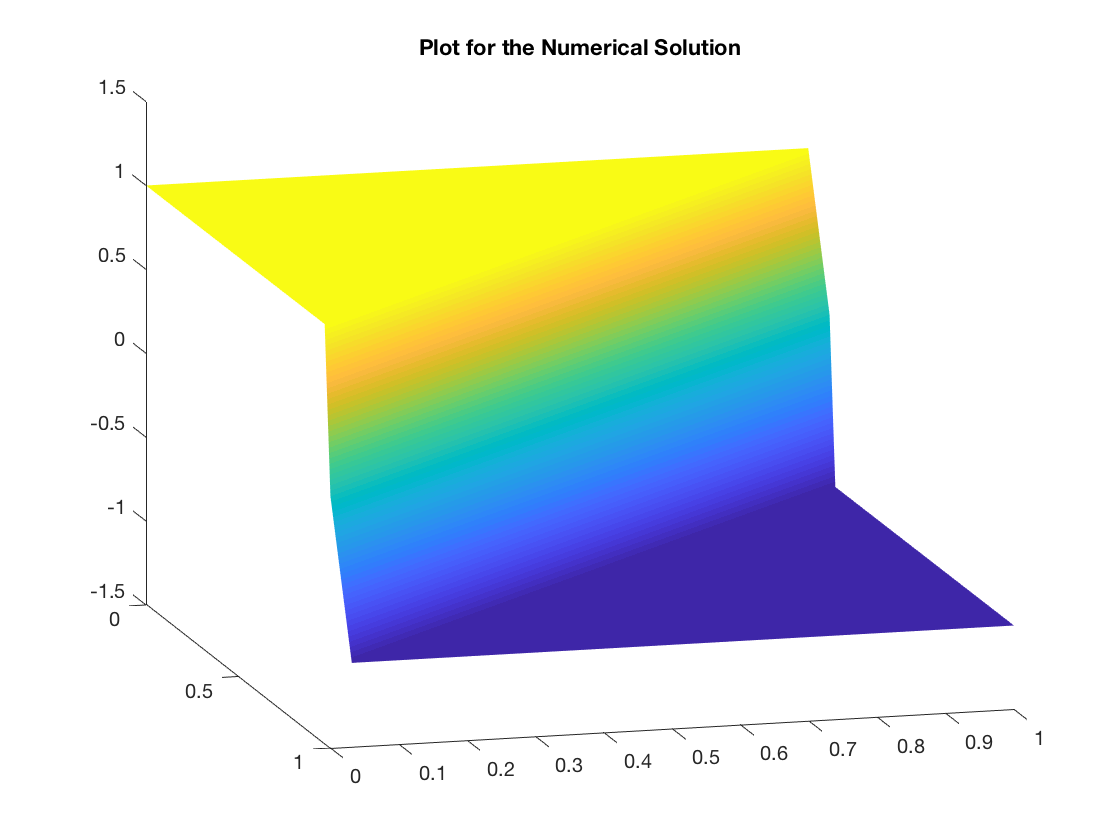}}
\resizebox{2.4in}{2.1in}{\includegraphics{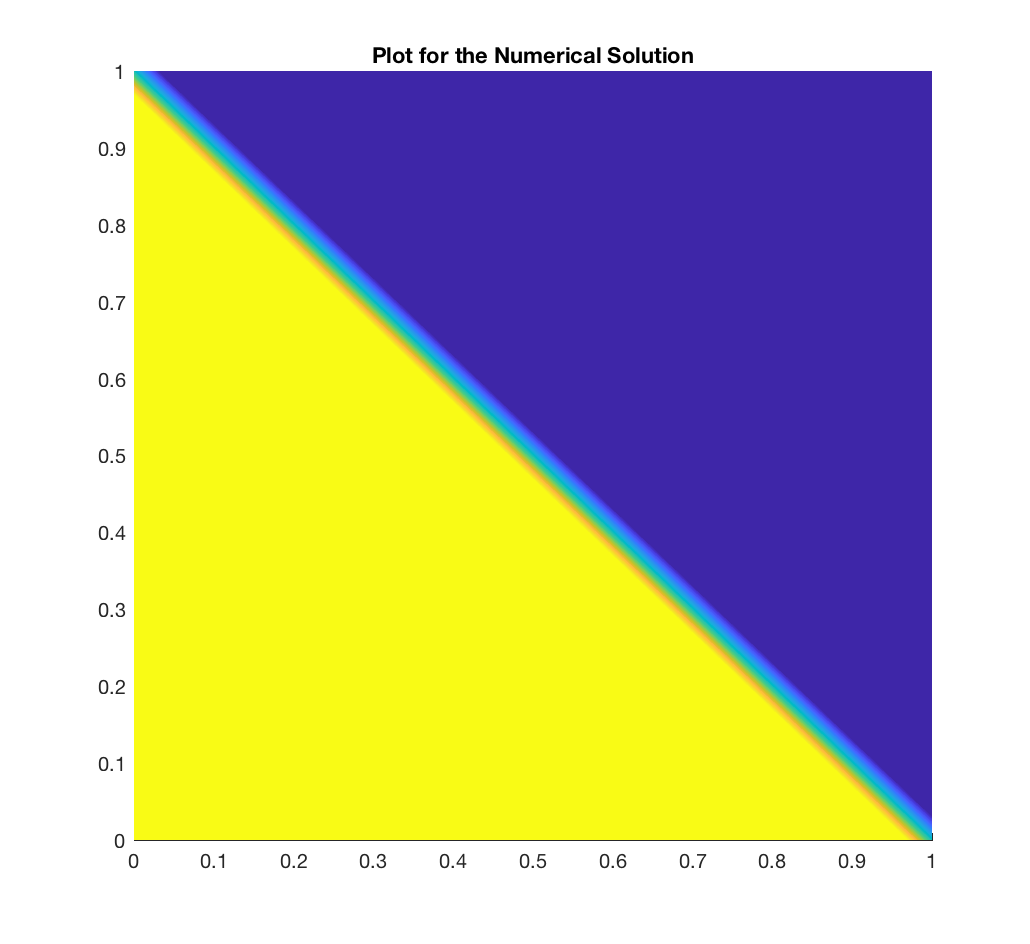}}
\end{tabular}
\caption{(with post-processing) Plots of the numerical solution $u_h$ on the unit square domain $\Omega_1$: $\bbeta=[1, -1]$, $c=0$, $\tau=0$, and discontinuous Dirichlet data at inflow boundary. Surface plot (left); contour plot (right).}
\label{u=1-2}
\end{figure}

In what follows of this section, all the plots are based on the post-processed approximations for the PD-WG solution $u_h$. Fig. \ref{f=1} illustrates the contour plots of the numerical solution $u_h$ on the unit square domain $\Omega_1$ with the following configuration: the convection vector $\bbeta(x, y)$ is piece-wisely defined such that $\bbeta(x, y)=[y+1, -x-1]$ if $y<1-x$ and $\bbeta(x, y)=[y-2, 2-x]$ otherwise, the reaction coefficient is $c=0$, and the inflow boundary data is given by $g=\cos(5y)$. The stabilization parameter is set as $\tau=0$. Fig. \ref{f=1} presents the contour plots for the post-processed numerical solution $u_h$ with the load function $f=1$ and $f=0$, respectively.

\begin{figure}[h]
\centering
\begin{tabular}{cc}
\resizebox{2.4in}{2.1in}{\includegraphics{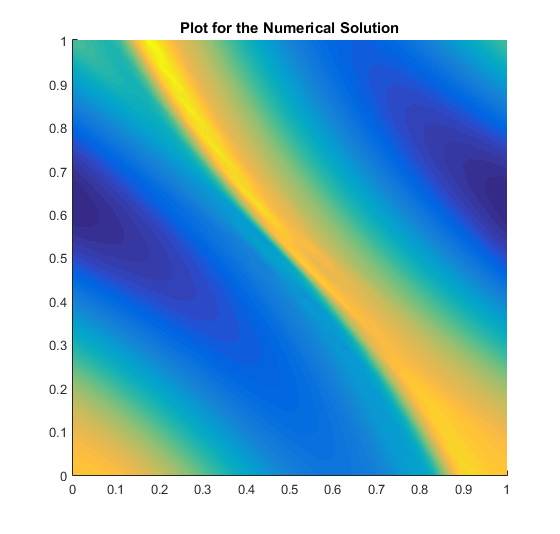}}
\resizebox{2.4in}{2.1in}{\includegraphics{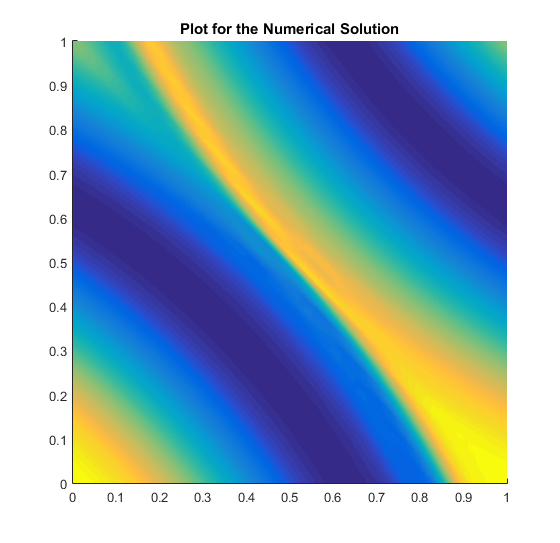}}
\end{tabular}
\caption{(with post-processing) Contour plots of the numerical solution $u_h$ on the unit square domain $\Omega_1$: $\bbeta(x, y)=[y+1, -x-1]$ if $y<1-x$ and $\bbeta(x, y)=[y-2, 2-x]$ elsewhere, $c=0$, the inflow boundary data $g=\cos(5y)$, $\tau=0$. The load function $f=1$ (left) and the load function $f=0$ (right).}
\label{f=1}
\end{figure}

Fig. \ref{f=1-C} shows the contour plots of the numerical solution $u_h$ on the unit square domain $\Omega_1$ for a test problem with the following configuration: the convection vector $\bbeta=[y-0.5, -x+0.5]$, the reaction coefficient $c=1$, the inflow boundary data $g=\cos(y)$, and the stabilization parameter $\tau=0$. The left one in Fig. \ref{f=1-C} demonstrates the contour plot of the numerical solution corresponding to the load function $f=10000$ and the right one is the contour plot of the numerical solution with load function $f=0$.

\begin{figure}[h]
\centering
\begin{tabular}{cc}
\resizebox{2.4in}{2.1in}{\includegraphics{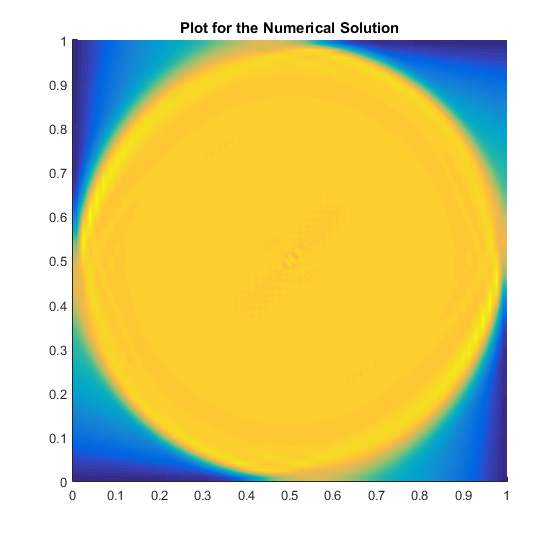}}
\resizebox{2.4in}{2.1in}{\includegraphics{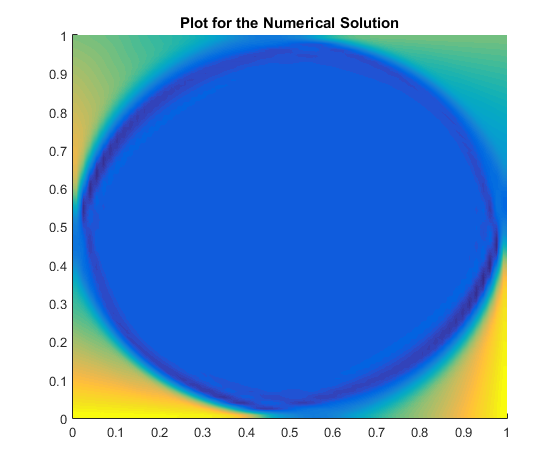}}
\end{tabular}
\caption{(with post-processing) Contour plots of the numerical solution $u_h$ on the unit square domain $\Omega_1$: $\bbeta=[y-0.5, -x+0.5]$, $c=1$, the inflow boundary data $g=\cos(y)$, $\tau=0$. The load function $f=10000$ (left) and the load function $f=0$ (right).}
\label{f=1-C}
\end{figure}

Fig. \ref{f=1-Cracked} shows the contour plots of the numerical solution $u_h$ resulting from the PD-WG scheme on the cracked square domain $\Omega_3$ when the load function is given by $f=0$ and $f=10000$. The configuration of this test problem is as follows: (1) the convection vector is $\bbeta=[y, -x]$, (2) the reaction coefficient is $c=0$, (3) the inflow boundary data is $g=\sin(x)$, and (4) the stabilization parameter is $\tau=0$.

\begin{figure}[h]
\centering
\begin{tabular}{cc}
\resizebox{2.4in}{2.1in}{\includegraphics{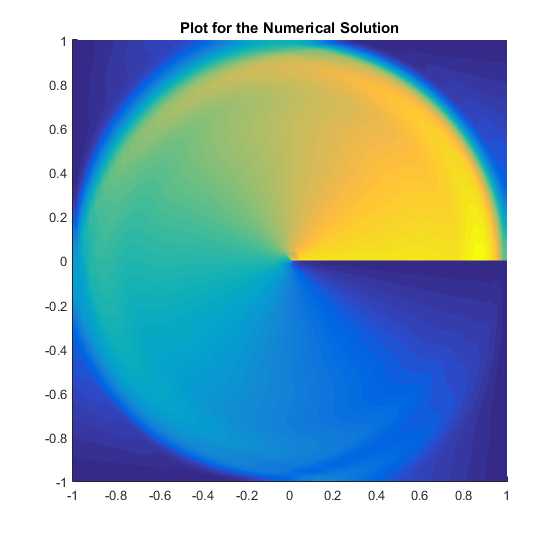}}
\resizebox{2.4in}{2.1in}{\includegraphics{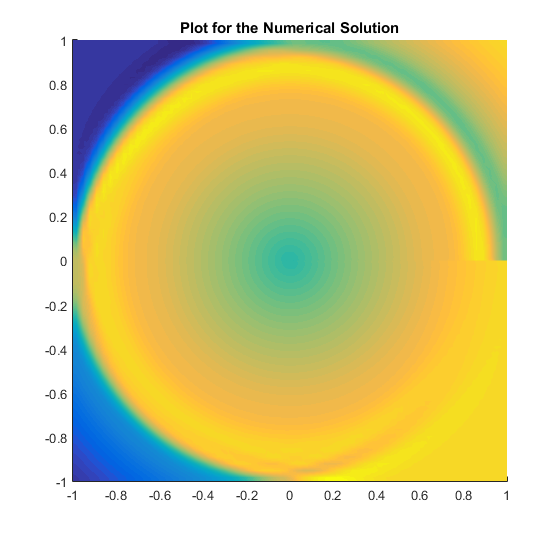}}
\end{tabular}
\caption{(with post-processing) Contour plots of the numerical solution $u_h$ on the cracked square domain $\Omega_3$: $\bbeta=[y, -x]$, $c=0$; the inflow boundary data $g=\sin(x)$, $\tau=0$. The load function $f=10000$ (left) and $f=0$ (right).}
\label{f=1-Cracked}
\end{figure}

Figure \ref{Lshaped} shows the contour plots of the numerical solution $u_h$ of the PD-WG method on the L-shaped domain $\Omega_2$. The test problem has the following configuration: (1) the load function is given by $f=0$ and $f=10000$, respectively, (2) the convection vector is $\bbeta(x, y)=[1, -1]$ for $y<1-x$ and $\bbeta(x, y)=[-1, 1]$ elsewhere, (3) the reaction coefficient is $c=1$, (4) the inflow boundary data is $g=\sin(x)\cos(y)$, and (5) the stabilization parameter is $\tau=0$. The left one in Fig. \ref{Lshaped} is the contour plot of the numerical solution corresponding to the load function $f=10000$, and the right one is the contour plot of the numerical solution for the load function $f=0$.

\begin{figure}[h]
\centering
\begin{tabular}{cc}
\resizebox{2.4in}{2.1in}{\includegraphics{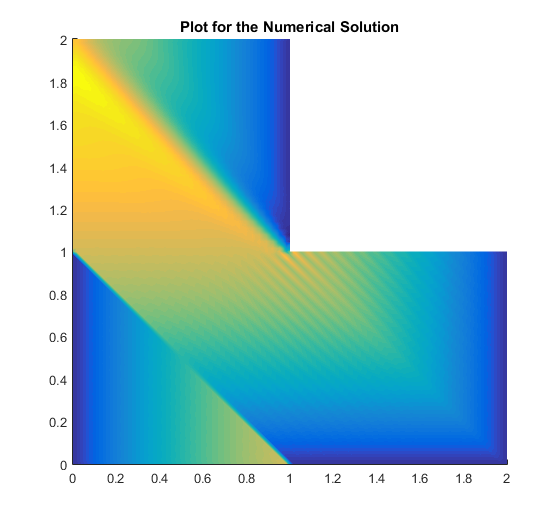}}
\resizebox{2.4in}{2.1in}{\includegraphics{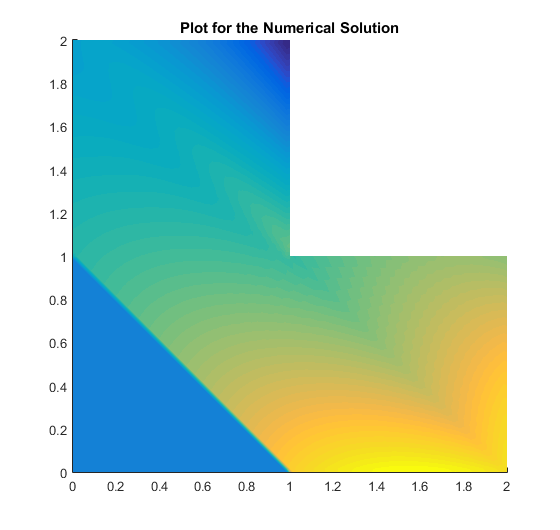}}
\end{tabular}
\caption{(with post-processing) Contour plots of the numerical solution $u_h$ on the L-shaped square domain $\Omega_2$: $\bbeta(x, y)=[1, -1]$ if $y<1-x$ and $\bbeta(x, y)=[-1, 1]$ elsewhere, $c=1$, $\tau=0$, and the inflow boundary data $g=\sin(x)\cos(y)$. The load function $f=10000$ (left) and $f=0$ (right).}
\label{Lshaped}
\end{figure}

In conclusion, the numerical performance of the PD-WG method (\ref{32})-(\ref{2}) for solving the linear transport problem \eqref{model} is consistent with or better than what the theory predicts in earlier sections of this paper. The numerical results clearly reveal a convergence of the method at the optimal order. We thus claim that the PD-WG finite element method is a stable, convergent, and practically useful numerical scheme for linear transport problems.

\vfill\eject


\begin{thebibliography}{99}

\bibitem{a2001} {\sc R. Abgrall}, {\em Toward the ultimate conservative scheme: following the quest}, J. Comput. Phys., 167 (2001), pp. 277-315.

 \bibitem{babuska} {\sc I. Babu\u{s}ka}, {\em The finite element method with Lagrange multipliers}, Numer. Math., vol. 20, pp. 179-192, 1973.

\bibitem{b1999} {\sc T. Barth}, {\em Numerical methods for gasdynamic systems on unstructured meshes, in An Introduction to Recent Developments in Theory and Numerics for Conservation Laws} (Freiburg/Littenweiler, 1997), Lect. Notes Comput. Sci. Eng. 5, Springer-Verlag, Berlin, 1999, pp. 195-285.

\bibitem{bd1999} {\sc T. Barth, and H. Deconinck, eds.}, {\em  High-order methods for computational physics}, Lect. Notes Comput. Sci. Eng. 9, Springer-Verlag, Berlin, 1999.

\bibitem{bc2001} {\sc P. Bochev, and  J. Choi}, {\em  Improved least-squares error estimates for scalar hyperbolic problems}, Comput. Methods Appl. Math., 1 (2001), pp. 115-124.

\bibitem{b1974} {\sc F. Brezzi}, {\em On the existence, uniqueness, and approximation of saddle point problems arising from Lagrange multipliers}, RAIRO, 8 (1974), pp. 129-151.

\bibitem{bcms} {\sc F. Brezzi, B. Cockburn, L. Marini, and E. Suli}, {\em Stabilization mechanisms in discontinuous Galerkin finite element methods}, Comput. Methods Appl. Mech. Eng. 195(25-28), 3293-3310 (2006).

\bibitem{bms} {\sc F. Brezzi, L. Marini, and E. Suli}, {\em Discontinuous Galerkin methods for first-order hyperbolic problems}, Math. Models Methods Appl. Sci. 14(12), 1893-1903 (2004).

\bibitem{bhm2000} {\sc W. Briggs, V. Henson, and S. McCormick}, {\em A Multigrid Tutorial}, 2nd ed., SIAM, Philadelphia, 2000.

 \bibitem{b} {\sc E. Burman}, {\em A unified analysis for conforming and nonconforming stabilized finite element methods using interior penalty}, SIAM J. Numer. Anal. 43(5), 2012-2033 (2005).

 \bibitem{ErikBurman-EllipticCauchy} {\sc E. Burman},
{\em Error estimates for stabilized finite element methods applied
to ill-posed problems}, C. R. Acad. Sci. Paris, Ser., vol. I 352,
pp. 655-659, 2014. http://dx.doi.org/10.1016/j.crma.2014.06.008

\bibitem{Burman2013} {\sc E. Burman}, {\em Stabilized finite element methods for nonsymmetric, noncoercive, and ill-posed problems. Part I: Elliptic
equations}, SIAM J. Sci. Comput., vol. 35, pp. 2752-2780, 2013.

\bibitem{burman2014} {\sc E. Burman},
{\em Stabilized finite element methods for nonsymmetric,
noncoercive, and ill-possed problems. Part II: hyperbolic
equations}, SIAM J. Sci. Comput, vol. 36, No. 4, pp.
A1911-A1936, 2014.

\bibitem{be} {\sc E. Burman, and A. Ern}, {\em Continuous interior penalty hp-finite element methods for advection and advection diffusion equations}, Math. Comp., 76, 1119-1140 (2007).

 \bibitem{bh} {\sc E. Burman, and P. Hansbo}, {\em Edge stabilization for Galerkin approximations of convection-diffusionreaction problems}, Comput. Methods Appl. Mech. Eng. 193(15–16), 1437-1453 (2004).

\bibitem{bs} {\sc  E. Burman, and B. Stamm}, {\em Discontinuous and continuous finite element methods with interior penalty for hyperbolic problems}, Tech. Report, EPFL-IACS report 17 (2005).

\bibitem{bs2007} {\sc  E. Burman, and B. Stamm}, {\em Minimal Stabilization for Discontinuous Galerkin Finite Element Methods for Hyperbolic Problems}, J Sci Comput (2007) 33: 183-208.

\bibitem{cj1988} {\sc G. Carey, and B. Jiang}, {\em  Least-squares finite elements for first-order hyperbolic systems}, Internat. J. Numer. Methods Engrg., 26 (1988), pp. 81-93.

\bibitem{c1999} {\sc B. Cockburn}, {\em Discontinuous Galerkin methods for convection-dominated problems}. In: High-Order Methods for Computational Physics. Lect. Notes Comput. Sci. Eng., vol. 9, pp. 69-224. Springer, Berlin
(1999).

\bibitem{c2001} {\sc B. Cockburn}, {\em Devising discontinuous Galerkin methods for non-linear hyperbolic conservation laws}, J.
Comput. Appl. Math. 128(1-2), 187-204 (2001).

\bibitem{cs} {\sc  B. Cockburn, and C. Shu}, {\em Runge-Kutta discontinuous Galerkin methods for convection-dominated problems}, J. Sci. Comput. 16(3), 173-261 (2001).

\bibitem{dd} {\sc J. Douglas, and T. Dupont}, {\em Interior penalty procedures for elliptic and parabolic Galerkin methods}, Computing Methods in Applied Sciences, Second Internat. Sympos., Versailles, 1975. Lecture Notes in
Phys., vol. 58, pp. 207–216. Springer, Berlin (1976).

\bibitem{eehj1996} {\sc K. Eriksson, D. Estep, P. Hansbo, and C. Johnson}, {\em  Computational Differential Equations}, Cambridge University Press, Cambridge, UK, 1996.

\bibitem{eg} {\sc  A. Ern, and J. Guermond}, {\em Discontinuous Galerkin methods for Friedrichs’ systems. i. General theory}, SIAM J. Numer. Anal. 44(2), 753-778 (2006).

\bibitem{hy2002} {\sc V. Henson, and U. Yang}, {\em BoomerAMG: A parallel algebraic multigrid solver and preconditioner}, Appl. Numer. Math., 41 (2002), pp. 155-177.

\bibitem{hjs2002} {\sc P. Houston, M. Jensen, and E. Suli}, {\em  Hp-discontinuous Galerkin finite element methods with least-squares stabilization}, J. Sci. Comput., 17 (2002), pp. 3-25.

\bibitem{hss} {\sc  P. Houston,  C. Schwab, and E. Suli}, {\em Discontinuous hp-finite element methods for advection-diffusionreaction problems}, SIAM J. Numer. Anal. 39(6), 2133-2163 (2002).

\bibitem{j2004} {\sc M. Jensen}, {\em Discontinuous Galerkin methods for Friedrichs’ systems with irregular solutions}. Ph.D. thesis, University of Oxford (2004).

\bibitem{jnp1984} {\sc C. Johnson, U. Navert, and J. Pitkaranta}, {\em Finite element methods for linear hyperbolic problems}, Comput. Methods Appl. Mech. Engrg., 45 (1984), pp. 285-312.

\bibitem{jp1986} {\sc C. Johnson, and J. Pitkaranta}, {\em  An analysis of the discontinuous Galerkin method for a scalar hyperbolic equation}, Math. Comp., 46 (1986), pp. 1-26.

\bibitem{lr} {\sc P. Lesaint, and P. Raviart}, {\em On a finite element method for solving the neutron transport equation}, Mathematical Aspects of Finite Elements in Partial Differential Equations, Proc. Sympos., Math. Res. Center, Univ. Wisconsin, Madison,Wis., 1974. Math. Res. Center, Univ. ofWisconsin-Madison, vol. 33, pp. 89-123. Academic, New York (1974).

\bibitem{l1992} {\sc R. LeVeque}, {\em Numerical Methods for Conservation Laws}, 2nd ed., Lectures in Mathematics ETH Zurich, Birkhauser Verlag, Basel, 1992.

\bibitem{lm1993} {\sc E. Lewis, and J. Miller}, {\em Computational Methods of Neutron Transport}, American Nuclear Society, La Grange Park, IL, 1993.

\bibitem{my} {\sc L. Mu, and X. Ye}, {\em A simple finite element method for linear hyperbolic problems}, Journal of Computational and Applied Mathematics 330 (2018) 330–339.

\bibitem{rh} {\sc W. Reed, and T. Hill}, {\em Triangular mesh methods for the neutron transport equation}, Tech. Report LAUR- 73-479, Los Alamos Scientific Laboratory (1973).

\bibitem{smmo2004} {\sc H. Sterck, T. Manteuffel, S. Mccormick, and L. Olson}, {\em Least-squares finite element methods and algebraic multigrid solvers for linear hyperbolic PDEs}, SIAM J. SCI. COMPUT.,  Vol. 26, No. 1, pp. 31-54.

\bibitem{tos2001} {\sc U. Trottenberg, C. Oosterlee, and A. Schuller}, {\em  Multigrid}, Academic Press, San Diego, 2001.

\bibitem{w2018} {\sc C. Wang}, {\em A New Primal-Dual Weak Galerkin
  Finite Element Method for Ill-posed Elliptic Cauchy Problems},
  submitted. arXiv:1809.04697.

\bibitem{ww2016} {\sc C. Wang, and J. Wang}, {\em A primal-dual weak Galerkin finite element method for second order elliptic equations in non-divergence form},  Mathematics of Computation, Math. Comp., vol. 87, pp. 515-545, 2018.

\bibitem{ww2017} {\sc C. Wang, and J. Wang}, {\em A Primal-Dual weak Galerkin finite element method for Fokker-Planck type equations}, arXiv:1704.05606, SIAM Journal of Numerical Analysis, accepted.

\bibitem{ww2018} {\sc C. Wang, and J. Wang}, {\em Primal-Dual Weak Galerkin Finite Element Methods for Elliptic Cauchy Problems}, submitted. arXiv:1806.01583.

\bibitem{wz2019} {\sc C. Wang, and L. Zikatanov}, {\em Low Regularity Primal-Dual Weak Galerkin Finite Element Methods for Convection-Diffusion Equations}, submitted. arXiv:1901.06743.

\bibitem{wy3655} {\sc J. Wang, and X. Ye}, {\em A weak Galerkin mixed finite element method for second-order elliptic problems}, Math. Comp, 83 (2014) 2101-2126.

\bibitem{yvb1998} {\sc I. Yavneh, C. H. Venner, and A. Brandt}, {\em  Fast multigrid solution of the advection problem with closed characteristics}, SIAM J. Sci. Comput., 19 (1998), pp. 111-125.

\end{thebibliography}
\end{document}